\documentclass[amssymb,12pt]{article}

\usepackage{mathtools}
\usepackage{geometry,amssymb,amsthm,amsmath,
graphics,enumerate}
\usepackage{times}
\usepackage{amsfonts}
\usepackage{amscd}
\usepackage{url}
\author{ Michael C.\
Laskowski  
and Danielle S.\ Ulrich
 \thanks{Both authors partially supported
by NSF grants DMS-1855789 and DMS-2154101.}
\\
Department of Mathematics\\University of Maryland
}

\def\includeE#1{{\lhook\kern-3.5pt\joinrel\smash{
    \mathop{\longrightarrow}\limits^{#1}}}}

\def\efor/{Example~\ref{E4}}

\def\BL/{Baldwin--Lachlan}
\def\Bu/{Buechler}
\def\Hr/{Hrushovski}
\def\lm/{locally modular}
\def\wm/{weakly minimal}
\def\nm/{non--modular}
\def\ss/{superstable}
\def\ud/{unidimensional}
\def\sm/{strongly minimal}

\def\abar{\overline{a}}

\def\bbar{\overline{b}}
\def\cbar{\overline{c}}
\def\dbar{\overline{d}}

\def\hbar{\overline{h}}

\def\rbar{\overline{r}}

\def\xbar{\overline{x}}

\def\zbar{\overline{z}}

\def\acl{{\rm acl}}

\def\dom{{\rm dom}}

\def\tp{{\rm tp}}

\def\tr/{trivial}
\def\nt/{non--trivial}
\def\st/{strong type}

\def\abar{\bar{a}}
\def\bbar{\bar{b}}
\def\cbar{\bar{c}}
\def\dbar{\bar{d}}

\def\phi{\varphi}
\def\A{{\cal A}}

\def\C{{\cal  C}}
\def\D{{\cal D}}
\def\F{{\cal F}}
\def\FF{{\bf F}}

\def\P{{\cal P}}

\def\tp{{\rm tp}}

\def\dom{{\rm dom}}
\def\acl{{\rm acl}}

\def\Fa0{{\FF^a_{\aleph_0}}}

\def\<{\langle}
\def\>{\rangle}

\newtheorem{Theorem}{Theorem}[section]
\newtheorem{Proposition}[Theorem]{Proposition}
\newtheorem{Definition}[Theorem]{Definition}
\newtheorem{Notation}[Theorem]{Notation}

\newtheorem{Lemma}[Theorem]{Lemma}
\newtheorem{Corollary}[Theorem]{Corollary}

\newtheorem{Fact}[Theorem]{Fact}
\def\iso{\cong}

\def\V{{\mathbb V}}

\def\dc{{\rm dc}}

\def\cc{{\bf c}}
\def\dd{{\bf d}}

\def\ss{{\bf s}}

\def\qftp{{\rm qftp}}
\def\ptl{{\rm ptl}}

\DeclareMathOperator{\Mod}{Mod}

\DeclareMathOperator{\Aut}{Aut}

\DeclareMathOperator{\CSS}{CSS}
\DeclareMathOperator{\css}{css}

\def\Edist{{\rm Edist}}
\def\hgt{{\rm ht}}
\def\CC{{\mathcal C}}
\def\BP{{\cal BP}}
\def\PP{{\mathbb P}}
\def\K{{\mathcal K}}
\def\LL{{\mathcal L}}
\def\MM{{\mathcal M}}

\def\MM{{\mathfrak{M}}}

\DeclareMathOperator{\Pos}{Pos}
\DeclareMathOperator{\Age}{Age}
\DeclareMathOperator{\Sym}{Sym}
\DeclareMathOperator{\ElDiag}{Eldiag}
\def\hhat{{\hat{h}}}
\def\alphabar{{\overline{\alpha}}}

\newcommand\myrestriction{\mathord\restriction}
\def\mr#1{\myrestriction_{#1}}

\begin{document}

	\title{Borel complexity of families of finite equivalence relations via large cardinals}
	
	\date{\today}

	\maketitle
	
	\begin{abstract}  
	We consider a large family of theories of equivalence relations, each with finitely many classes, and assuming the existence of an $\omega$-Erd\H{o}s cardinal, we determine which of these theories are Borel complete.  
 We develop machinery, including {\em forbidding nested sequences} which implies a tight upper bound on Borel complexity, and {\em admitting cross-cutting absolutely indiscernible sets} which in our context implies Borel completeness.
 In the Appendix we classify the reducts of theories of refining equivalence relations, possibly with infinite splitting.  
	\end{abstract}

 \section{Introduction}

For many years, the authors have sought to identify dividing lines for Borel complexity of invariant classes of countable structures.  One of the stumbling blocks has been identifying how to handle types in stable theories that have a perfect set of strong types extending it.  Indeed, in \cite{URL}, along with Richard Rast, the authors proved that REF(bin), the theory of binary splitting, refining equivalence relations, is not Borel complete (in fact, `countable sets of countable sets of reals' do not Borel embed into $\Mod(REF(bin))$ yet the isomorphism relation on pairs of countable models is not Borel.  It became apparent that many of the existing tools of Descriptive Set Theory would not be applicable to such theories.

In order to isolate the strong type problem from other phenomena, we concentrate on families of theories whose model theory is extremely tame.  For almost all of this paper, we investigate theories of countably many equivalence relations, each with finitely many classes.   Any such theory is mutually algebraic, equivalently, every completion is weakly minimal with trivial geometries (see e.g., \cite{MA}). Moreover, in our examples, for any model $M$, $\acl(X)=X$ for every subset $X\subseteq M$ and there is a unique 1-type.
Somewhat surprisingly, we find that even here, understanding the Borel complexity of such a theory is extremely involved -- so much so that in some cases we are only able to prove non-Borel completeness by using a large cardinal axiom.

Why should the situation be so complicated?  With Proposition~4.4 of \cite{MA}, the first author showed that the {\em elementary diagram} $\ElDiag(M)$ of any model of a mutually algebraic theory admits a strong structure theory.  From this, it follows easily that for any mutually algebraic theory $T$, the elementary diagram $\ElDiag(M)$ has a Borel isomorphism relation and there is no
Borel embedding of `countable sets of countable sets of reals' into the class of countable models of $\ElDiag(M)$.  But, passing from $Th(M)$ to $\ElDiag(M)$ obviates the behavior we were aiming to study.  For $M$ mutually algebraic, the strong type structure on $\ElDiag(M)$ is degenerate.  

The Borel complexity of mutually algebraic structures is controlled by the profinite group of elementary permutations of $\acl^{eq}(\emptyset)$.  In previous works, we carefully analyzed $REF(\delta)$, refining equivalence relations in which each $E_n$-class is partitioned into $\delta(n+1)$ $E_{n+1}$-classes.  Here, we prove that whereas $REF(\delta)$ is not Borel complete, there is a {\em tame expansion}  (see Definition~\ref{tamedef}) of some countable model that is Borel complete.    As well,  in the Appendix we prove that no model $M$ of the theory of refining equivalence relations has any interesting reducts, from which it follows that no reduct of $REF(\delta)$ is Borel complete either.
This last remark is rather peculiar.  If one starts with a countable model 
$M\models REF(\delta)$, neither its theory nor any of the theories of its reducts are Borel complete.  Moreover, if one adds constants naming each point and thereby obtaining $\ElDiag(M)$, this expanded structure is not Borel complete.  However, the expansion $\ElDiag(M)$ does have a Borel complete reduct, say by Theorem~3.2 of 
\cite{reducts}.

By contrast, in a pair of papers \cite{reducts} and \cite{expansions}, we give a description of the countable models of  $CC(\delta)$, the theory of countably many, cross-cutting equivalence relations $E_n$, where each $E_n$ has $\delta(n)<\omega$ equivalence relations.  On one hand, if 
$\{\delta(n):n\in\omega\}$ is unbounded, then $CC(\delta)$ is Borel complete by Theorem~2.1 of \cite{reducts}.  By contrast, if $\delta(n)<m$ for some fixed $m\in\omega$, then $CC(\delta)$ is far from being Borel complete by Theorem~6.2 of \cite{expansions}.
With Corollary~\ref{introCor} here, we generalize both of these results.  
However, so long as $\delta(n)\ge 2$ for each $n$, there exist countable models $M\models CC(\delta)$ that have Borel complete reducts.

To try to develop a general framework for handling the behavior of strong types, in Definition~\ref{start} we let
 $\mathbb{P}$ be the class of all triples $(P, \leq, \delta)$ where $(P, \leq)$ is a countable poset for which
 $\{q \in P: q \leq p\}$ is finite for all $p\in P$,
 and where $\delta: P \to \omega \backslash \{0, 1\}$. When $\leq$ and $\delta$ are clear from context then we omit them.  In this context, $REF(\delta)$ is obtained by taking $(P,\le)$ to be the chain $(\omega,\le)$.  
 By contrast, $CC(\delta)$ is obtained by taking $(P,\le)$ to be a countably infinite antichain (no two distinct elements are comparable).  
 
 Say that $(P, \leq, \delta)$ is {\em bounded} if there is some $m < \omega$ such that $\delta \leq m$ and such that every chain from $P$ has length at most $m$.

 Say that $(P, \leq, \delta)$ is {\em minimally unbounded} if $P$ is unbounded, and yet for all  downward closed $Q\subseteq P$, either $Q$ or $P \backslash Q$ is bounded.

To each $(P, \leq, \delta) \in \PP$ we associate a first-order theory $T(P, \leq, \delta) = T_P$ in the language $\LL_P = \{E_p: p \in P\}$, where $E_p$ refines $E_q$ for $p > q$ and $\delta$ determines how many $E_p$-classes each $E_q$-class splits into. In the main case of interest when $P$ is infinite, $T_P$ eliminates quantifiers and is mutually algebraic (weakly minimal trivial).  Moreover, for every $M\models T_P$, $\acl(X)=X$ for all subsets $X\subseteq M$ and `$x=x$' generates a complete 1-type.

The following three theorems classify when $T(P, \leq, \delta)$ is Borel complete. By a {\em tame expansion} of $T_P$ we mean an expansion where every new symbol $S$ is relational, and is furthermore $E_p$-invariant for some $p \in P$.
 \begin{Theorem}\label{Intro1}
     If $P$ is bounded then $T_P$ is not Borel complete.  In fact, there is no Borel embedding of `countable sets of countable sets of reals' into any tame expansion of $T_P$. Conversely, if $P$ is unbounded then some tame expansion of $T_P$ is Borel complete.
 \end{Theorem}

 \begin{Theorem}\label{Intro2}
 If $P$ is unbounded, but not minimally unbounded, then $T_P$ is Borel complete.
 \end{Theorem}

For the following, $\kappa(\omega)$ is the first $\omega$-Erd\H{o}s cardinal (see Definition~\ref{Erdos}). We quote machinery of the second author from \cite{SB}, where it is shown that if $\kappa(\omega)$ exists, then various Schr\"{o}der-Bernstein properties imply the failure of Borel completeness. It is open whether the large cardinal is necessary.
 \begin{Theorem}\label{Intro3}
     Suppose $P$ is minimally unbounded, and assume $\kappa(\omega)$ exists. Then $T_P$ is not Borel complete.
 \end{Theorem}

 The proofs of these theorems involve combining some tools from Descriptive Set Theory.  Some of these are novel and are discussed in Section 2, which is entirely general.  Section~3 discusses {\em down-finite} partial orders, which are the only partial orders we consider.  The assumption is natural as we require  the relations $E_p$ to have only finitely many classes.
  The classes of theories $T_P$ are defined by Section~4
 and this framework is used until the end of the paper.


 The appendix is self-contained and does not require the equivalence relations to have finite splitting.  The main result,
 Theorem~\ref{app}, is that any reduct of any model $M\models REF$ is (after reindexing) also a model of $REF$.  
 It is hoped that this result may be of interest in its own right.  

 \section{Borel complexity and dividing lines from variants of indiscernibility}

In their seminal paper \cite{FS}, Friedman and Stanley define and develop a notion of {\em Borel reducibility} among 
first order and infintarily definable classes of structures  with universe $\omega$ in  countable languages.  The classes $\Mod(\Phi)$ for $\Phi\in L_{\omega_1,\omega}$
are precisely those   classes that are 
 Borel and invariant under permutations of $\omega$.  
 Such classes include $\Mod(T)$, the countable models of first order theories $T$.      For such sentences $\Phi$, $\Psi$,  possibly in different languages, a  {\em Borel reduction} is a Borel function $f:\Mod(\Phi)\rightarrow \Mod(\Psi)$ that satisfies $M\iso N$ if and only if $f(M)\iso f(N)$.  One says that $\Phi$ is {\em Borel reducible} to $\Psi$, $\Phi\le_B\Psi$,    if there is a Borel reduction $f:\Mod(\Phi)\rightarrow \Mod(\Psi)$.  $\Phi$ and $\Psi$ are {\em Borel equivalent} if $\Phi\le_B\Psi$ and $\Psi\le_B \Phi$.  
Among Borel invariant classes there is a maximal class with respect to $\le_B$.  We say $\Phi$ is {\em Borel complete} if $\Mod(\Phi)$ is in this maximal class.  Examples include the (incomplete) theories of graphs, linear orders, groups, and fields.

Many of the classical tools of Descriptive Set Theory do not help in this context, as the isomorphism relation on pairs of 
models in $\Mod(T)$ is typically not Borel.  Here, we adapt various tools from model theory to aid in identifying the Borel complexity of theories.

 Indiscernible sequences inside models of a theory
 or an infinitary sentence are important in determining the model-theoretic complexity of 
 its  class of models.  
 However, in the first order context, merely asking for the existence of a model with a non-constant, infinite sequence of
 indiscernibles is not a strong requirement.  Indeed, a compactness argument shows that any theory $T$ admitting an infinite model also admits models containing non-constant indiscernible sequences of any order type.  
 This extreme freedom may fail for  sentences $\Phi\in L_{\omega_1,\omega}$, but the existence of a non-constant, infinite indiscernible sequence inside some model does not imply that the class of countable models of $\Phi$ is complicated.
 Indeed, examples of Leo Marcus~\cite{Marcus} and Julia Knight~\cite{Knightindisc} show that a sentence $\Phi\in L_{\omega_1,\omega}$ can have a unique model $M$, yet $M$ contains an infinite,  fully indiscernible  subset.

 However, both weakenings and strengthenings of indiscernibility  impact Borel complexity.
 Recall that in any structure $M$, a sequence $(a_n:n\in\omega)$ is indiscernible iff it is nested (see Definition~\ref{nested}) and $\tp(a_m/A_{<m})$ does not split over $A_{<n}$ for every $n<m<\omega$.
 We concentrate on the first of these two conditions. The concept of non-splitting will not play a role in this paper.

\begin{Definition}   \label{nested} {\em  Say $M$ is any $\LL$-structure and let $(a_n:n\in\omega)$ be any sequence of singletons from $M$.  As notation, for each $n\in\omega$, let
$A_{<n}=\{a_j:j<n\}$ and let $q_n(x)=\qftp(a_n/A_{<n})$.  Call the sequence  {\em nested} if $q_n(x) \subseteq q_m(x)$ for every $n \leq m$. Alternatively, $(a_n: n \in \omega)$ is nested if for each $n$, $(a_i:i < n, a_n)$ and $(a_i: i < n, a_{n+1})$ have the same quantifier-free type.
We say that {\em $\Phi$ forbids nested sequences} if, for every $M\models \Phi$, every nested sequence from $M$ is eventually constant, i.e.,
for some $n$, $a_m=a_n$ for every $m\ge n$.  We say {\em $\Phi$ admits nested sequences} if the negation holds.  
}
\end{Definition}

This is an extremely restrictive condition on $\Phi$.  By compactness, we see that no first order theory $T$ with an infinite model forbids nested sequences.
We note two consequences of this property.  The first connects to classical  notions of $\alpha$-back-and-forth systems, at least when $\alpha=1$.

\begin{Definition}  {\em  Suppose $M,N$ are $\LL$-structures and $\abar\in M^n$,
$\bbar\in N^n$.  We say $(M,\abar)\equiv_0 (N,\bbar)$ if and only if $\qftp(\abar)=\qftp(\bbar)$;

For any ordinal $\alpha$, $(M,\abar)\equiv_{\alpha+1} (N,\bbar)$ if, for every $a^*\in M$ there is $b^*\in N$ such that
$(M,\abar a^*)\equiv_\alpha (N,\bbar, b^*)$ and vice versa; and 

For any non-zero limit ordinal $\delta$, $(M,\abar)\equiv_{\delta} (N,\bbar)$ if and only if $(M,\abar)\equiv_{\alpha} (N,\bbar)$ for every $\alpha<\delta$.
}
\end{Definition}

Here, we are only concerned with the case $\alpha=1$.

\begin{Definition} \label{SBembed} {\em  

We say $M\subseteq^*_1 N$ if $M$ is a substructure of $N$ such that, for all $\abar\in M^{<\omega}$, $(M,\abar)\equiv_1 (N,\abar)$.
A {\em 1-embedding} $f:M\preceq^*_1 N$ is an $\LL$-embedding for which $f(M)\subseteq^*_1 N$.
}
\end{Definition} 

It is easily seen that for $M\subseteq N$, $M\subseteq^*_1 N$ if and only if for every $\abar\in M^{<\omega}$ and every $b\in N$, there is $a^*\in M$
such that $\qftp(\abar a^*)=\qftp(\abar b)$ if and only if `$M$ is relatively $\omega$-saturated in $N$ for quantifier-free types.'


\begin{Proposition}  \label{useforbid}    Suppose $f:M\preceq^*_1 N$, where $M$ forbids nested sequences.  Then $f$ is onto.
\end{Proposition} 

\begin{proof} We can suppose $M \subseteq_1^* N$. By way of contradiction, suppose there is some $b\in N\setminus M$.  We construct a non-constant, nested sequence $(a_n:n\in\omega)\subseteq M$
as follows.  Suppose $(a_i:i<n)\subseteq M$ have been chosen.  As $f(M)\subseteq^*_1 N$, choose $a_n\in M$ such that $\qftp (a_i:i\le n)=\qftp (a_i:i<n)b$.
As $b$ realizes $\tp(a_n/A_{<n})$ for each $n$, we have $q_n(x)\subseteq q_m(x)$ for all $m\ge n$, so the sequence $(a_n:n\in\omega)$ is nested.  Since $b\not\in M$,
the sequence is non-constant as well.
\end{proof}

In \cite{GaoCLI}, Su Gao investigated automorphism groups $\Aut(M)$ that admit a compatible, left-invariant, complete metric.  Such groups are now called {\em cli}.
Gao~\cite{GaoCLI}
proves that for a countable structure $M$, $\Aut(M)$ is cli if and only if $M$ has no proper $L_{\omega_1,\omega}$-elementary extension if and only if every $L_{\omega_1,\omega}$-embedding
$f:M\rightarrow N$ is onto.  

\begin{Corollary}   \label{forbidcli}  If $\Phi$ forbids nested sequences of types, then $\Aut(M)$ is cli for  every countable $M\models\Phi$.
\end{Corollary}

\begin{proof}  By way of contradiction, suppose there were some $M\models\Phi$ with a proper, $L_{\omega_1,\omega}$ elementary extension $N\succ M$.
In particular, $M\subseteq_1^* N$, which contradicts Proposition~\ref{useforbid}.  Thus, $\Aut(M)$ is cli for every $M\models \Phi$ by Gao's theorem.
\end{proof}

We remark that  $M$ forbidding nested sequences is strictly stronger than $\Aut(M)$ being cli.   Indeed, the examples of Marcus and Knight mentioned above are
cli, but visibly do not forbid nested sequences. 

\medskip

In \cite{URL}, we investigate the Borel complexity of $\Phi$ by considering its {\em potential canonical Scott sentences} $\phi\in \CSS(\Phi)_\ptl$.  
These are the sentences $\phi$ of $L_{\infty,\omega}$ that become a canonical Scott sentence of some countable model $M$ in some (equivalently, any) forcing extension $\V[G]$
for which $\phi\in (L_{\omega_1,\omega})^{\V[G]}$.  The number of these sentences $||\Phi||:=|\CSS(\Phi)_\ptl|$  (possibly a proper class) measures  the Borel complexity   of 
$\Mod(\Phi)$, the class of countable models of $\Phi$.  
Theorem~3.10(2) of \cite{URL} states that if $\Phi$ is Borel reducible to $\Psi$, then $||\Phi||\le||\Psi||$.

If $\phi$ is a canonical Scott sentence -- that is, $\phi\in \CSS(\Phi)_\ptl$ -- then we can identify $\phi$ with the set $S^{<\omega}_\infty(\phi):=\{\css(M,\abar):\abar\in M^{<\omega}\}$
where $M$ is some (equivalently, any)  countable $M\models\phi$ occurring in any forcing extension $\V[G]$ in which $\phi\in( L_{\omega_1,\omega})^{\V[G]}$.  As this set does not depend
on the choice of either $G$ or $M$, it follows by the product forcing lemma (see e.g., Lemma~2.5 of \cite{URL})  that the set $S^{<\omega}_\infty(\phi)\in \V$.
For each $n\in\omega$, we refer to elements of $S^n_{\infty}(\phi)$ as infinitary types $p(x_0,\dots,x_{n-1})$.  When discussing such types, the following notation will be helpful.

\begin{Notation} \label{swap} {\em  For any $n\ge 1$ and any  type  $r(x_0,\dots,x_{n})$, we distinguish two associated types $\pi(r)(x_0,\dots,x_{n-1})$ and $\pi^*(r)(x_0,\dots,x_{n-1})$.  $\pi(r)$ is simply the projection $r\mr{x_0,\dots,x_{n-1}}$ onto the first $n$ coordinates, while $\pi^*(r)$ is obtained by first swapping the roles of $x_{n-1}$ and $x_n$, and then taking the projection.
That is, for any formula $\delta$,
$$\delta(x_0,\dots,x_{n-2},x_{n-1})\in\pi^*(r) \quad \hbox{if and only if} \quad \delta(x_0,\dots,x_{n-2},x_n)\in r$$
}
\end{Notation}

As an example of this usage, for any $\LL$-structure $M$ and any $\omega$-sequence $(a_n:n\in\omega)$ from $M$,
the sequence $(a_n:n\in\omega)$ is nested if and only if $\pi(r_n)=\pi^*(r_n)$ for every $n\ge 1$, where $r_n:=\qftp(a_0,\dots,a_n)$.

 Note that $S^{<\omega}_\infty(\phi)$ has 
an `amalgamation' property:  for every $n\in\omega$, every infinitary type $s(x_0,\dots,x_{n-1})\in S^n_\infty(\phi)$ and  all extensions $p(x_0,\dots,x_n), q(x_0,\dots,x_n)\in S^{n+1}_\infty(\phi)$, there is at least one 
$r(x_0,\dots,x_{n+1})\in S^{n+2}_\infty(\phi)$ such that $\pi(r)=p$ and $\pi^*(r)=q$.
Furthermore,  if $p\neq q$, then $(x_n\neq x_{n+1})\in r$ for any such amalgam $r$.

To see this, fix any forcing extension $\V[G]$ in which $\phi$ is hereditarily countable and, in $\V[G]$, choose any countable $M\models \phi$. Now, given $s,p,q$ as above, since $p,q$ extend $s$,
there are $\abar\in M^n$, $b,c\in M$ such that $p=\css(M,\abar,b)$ and $q=\css(M,\abar,c)$.  Then $r=\css(M,\abar,b,c)$ is an amalgam.

\begin{Theorem}\label{topTrivialTheorem}
Suppose $\Phi\in L_{\omega_1,\omega}$ forbids nested sequences.  Then:
\begin{enumerate}
\item   Every $\phi\in \CSS_\ptl(\Phi)$ is a sentence of $\mathcal{L}_{(2^{\aleph_0})^+, \omega}$;
\item The Scott rank  $SR(\phi) < (2^{\aleph_0})^+$; 
\item  $||\Phi||\le\beth_2$;
\item  `Countable sets of countable sets of reals' do not Borel embed into $\Mod(\Phi)$.
\end{enumerate}
Moreover, since $\Phi$ forbidding nested sequences is preserved under expansions of the language, the same results hold for any expansion of $\Phi$.
\end{Theorem}

\begin{proof}  (2) follows from (1) by the definitions, and (3) follows from (1) since $\LL$ is countable.  (4) follows from (3) by Theorem~3.10(2) of \cite{URL}, since
`countable sets of countable sets of reals' has $\beth_3$ potential canonical Scott sentences.

So, it suffices to prove (1).  
By Proposition 4.6 of \cite{URL},  it is enough to show that $|S^{<\omega}_\infty(\phi)| \leq 2^{\aleph_0}$. 
Suppose towards a contradiction this were not the case.  Choose a forcing extension $\V[G]$ in which $\phi\in HC^{\V[G]}$ and, in $\V[G]$, choose a countable $M\models\phi$.
By Shoenfield absoluteness, $\Phi$ also forbids nested sequences in $\V[G]$, so we will obtain a contradiction by constructing a nested sequence $(a_n:n\in\omega)$ of distinct elements inside $M$.

Working in $\V$, as we are assuming $|S^{<\omega}_\infty(\phi)|>\beth_1$, and
since $S^0_\infty(\phi)$ is a singleton, 
choose a $k\in\omega$ and $s(x_0,\dots,x_{k-1})\in S^k_\infty(\phi)$ that has $>\beth_1$ extensions to $S^{k+1}_\infty(\phi)$.  
From this, we recursively construct sets $X_m\subseteq S^{k+m+1}_\infty(\phi)$ and types $p_m\in X_m$ satisfying:
\begin{enumerate}
\item  $|X_m|>\beth_1$;
\item  Any $p,q\in X_m$ have the same  quantifier-free type;
\item   For every $r\in X_{m+1}$, $\pi(r)=p_m$ and 
$\pi^*(r)\in X_m$, but $\pi(r)\neq \pi^*(r)$. (In particular $p_m \subseteq p_{m+1}$.)
\end{enumerate}

To begin the construction,  as there are only $\beth_1$ quantifier-free types in $S^{k+1}_\infty(\phi)$, 
by our choice of $s$,
choose a subset $X_0$ of extensions of $s$, all with the same quantifier-free type and such that $|X_0| > \beth_1$, and choose $p_0\in X_0$ arbitrarily.
Now, assume $X_m$ and $p_m$ have been chosen.  For each $q\in X_m\setminus\{p_m\}$, choose an amalgam $r_q\in S^{k+m+2}_\infty(\phi)$ such that $\pi(r_q)=p_m$ and $\pi^*(r_q)=q$. To see this is possible, note that when $m = 0$, both $p_m$ and $q$ extend $s$, and when $m > 0$, both $p_m$ and $q$ extend $p_{m-1}$. Now let $Y_{m+1}=\{r_q:q\in X_m\setminus\{p_m\}\}$.
Since the map $q\mapsto r_q$ is injective, $|Y_{m+1}|>\beth_1$, so choose a subset $X_{m+1}\subseteq Y_{m+1}$,
all of whom have the same quantifier-free type, of size $>\beth_1$.


Now, forgetting about the witnessing sets $X_m$, we have constructed a sequence $(p_m\in S^{k+m+1}_\infty(\phi):m\in\omega)$ such that each extends $s$;
$p_m\subseteq p_{m+1}$;  $p_m(\xbar,y_i:i\le m)\vdash \bigwedge_{i<j\le m} y_i\neq y_j$; and the types $p_m=\pi(p_{m+1})$ 
and $\pi^*(p_{m+1})$ have the same quantifier free type.  

Finally, we pass this data to $\mathbb{V}[G]$, recalling that every type in $S^{<\omega}_\infty(\phi)$ is of the form $\tp_M^\infty(\cbar)$ for some $\cbar\in M^{<\omega}$.
Choose $\bbar\in M^k$ so that $s=\tp_M^\infty(\bbar)$ and recursively choose $a_m\in M$ so that $p_m=\tp_M^\infty(\bbar,a_i:i\le m)$.
Then $(a_n:n\in\omega)$ is a nested sequence of distinct elements of $M$, giving our contradiction.
\end{proof}

On the stronger side, one can talk ask whether a sentence $\Phi\in L_{\omega_1,\omega}$ contains models with {\em arbitrarily long} non-constant sequences.   By Theorem~5.6 of \cite{expansions}, this property implies
that $S_\infty$ divides $\Aut(M)$ for some $M\models \Phi$, which is a notion introduced by
 Hjorth~\cite{Hjorth}.  

\begin{Definition}  {\em For topological groups $G,H$, we say $H$ {\em divides} $G$ if there is a closed subgroup $G^*\le G$
and a continuous, surjective homomorphism $\pi:G^*\rightarrow H$.}
\end{Definition}

Of special interest is when $G=\Aut(M)$ for some countable $\LL$-structure and $H=S_\infty= \Sym(\omega)$.

With Theorem~5.5 of \cite{expansions}, we give many equivalents of a sentence $\Phi\in L_{\omega_1,\omega}$ having
a model $M$ for which $S_\infty$ divides $\Aut(M)$.  Shaun Allison considers this notion locally, i.e., when $G=\Aut(M)$
is fixed, and gives several other equivalents.  We isolate a notion that was unnamed, but  present in both \cite{URL} and \cite{modules}, that gives yet another equivalent to $S_\infty$ dividing $\Aut(M)$.


\begin{Definition} \label{abs} {\em A countable structure $M$ {\em admits absolutely indiscernible sets} if there are disjoint subsets $\{D_n:n\in\omega\}$ of $M$
such that, for every permutation $\sigma\in \Sym(\omega)$, there is an automorphism $\sigma^*\in\Aut(M)$ such that $\sigma^*[D_n]=D_{\sigma(n)}$ for 
every $n\in\omega$.  For $\Phi\in L_{\omega_1,\omega}$ we say {\em $\Phi$ admits absolutely indiscernible sets} if some countable $M\models \Phi$ does.
}
\end{Definition}

The hard direction of the following Fact is implicit in the proof of Allison's 
Theorem~3.7 from \cite{Allison}.

\begin{Fact} \label{Allisonthm}
Let $M$ be any countable structure.  Then $S_\infty$ divides $\Aut(M)$ if and only if $M$ admits absolutely indiscernible sets.
\end{Fact}  

\begin{proof}  (Easy direction).  Suppose $\{D_n:n\in\omega\}$ is a family of absolutely indiscernible sets from 
a countable $\LL$-structure $M$.
Let $\LL'=\LL \cup \{U,E\}$ let $M'$ be the expansion of $M$ formed by interpreting $U$ as $\bigcup\{D_n:n\in\omega\}$ and
$E$ as the equivalence relation on $U$ given by $E(a,b)$ iff $a,b$ are in the same $D_n$.
Then $\Aut(M')$ is a closed subgroup of $\Aut(M)$ and, as any $f\in\Aut(M')$ permutes the sets $\{D_n:n\in\omega\}$,
we get an induced map $\pi:\Aut(M')\rightarrow S_\infty$.  Clearly, $\pi$ is a continuous homomorphism, and
since $\{D_n\}$ is absolutely indiscernible, $\pi$ is a surjection.  
\end{proof}

By Theorem~5.5 of \cite{expansions}, we know that any $\Phi\in L_{\omega_1,\omega}$ admitting absolutely indiscernible sets has a Borel complete expansion, but for a given $\Phi$, it is of interest to know how 
much additional structure must be added to obtain a Borel complete class.
The following is a strengthening of Definition~\ref{abs}.

\begin{Definition} \label{cross-cutting} {\em A countable structure $M$ {\em admits cross-cutting absolutely indiscernible sets}
if there are $\Aut(M)$-invariant equivalence relations $E_0,E_1$ and families of sets $\D_*^0=\{D^0_n:n\in\omega\}$,
$\D^1_*=\{D^1_m:m\in\omega\}$ satisfying
\begin{enumerate}
    \item  For all $a\in\D^0_*$, $b\in D^1_*$, there is $c\in M$ such that $M\models E_0(c,a)\wedge E_1(c,b)$;
    \item  For all distinct $n\neq n'$, $a\in D^0_n$, $a'\in D^0_{n'}$ implies $M\models \neg E_0(a,a')$ (and dually for 
    $\D^1_*$ and $E_1$.)
    \item  For all pairs $\sigma_0,\sigma_1\in \Sym(\omega)$ there is some $\tau\in\Aut(M)$ such that for all
    $n,m\in\omega$, $\tau[D^0_n]=D^0_{\sigma_0(n)}$ and $\tau[D^1_m]=D^1_{\sigma_1(m)}$.
\end{enumerate}
We say that $\Phi\in L_{\omega_1,\omega}$ {\em admits cross-cutting absolutely indiscernible sets} if some countable
$M\models \Phi$ does.
}
\end{Definition}

By itself, $\Phi$ admitting cross-cutting absolutely indiscernible sets does not imply Borel completeness.
Indeed the complete first-order theory $T$ of cross-cutting equivalence relations $E_0,E_1$, each with infinite splitting
and with $[a]_{E_0}\cap [b]_{E_1}$ infinite for every $a,b$ is $\omega$-categorical, yet admits cross-cutting, absolutely indiscernible sets.\footnote{Given $M\models T$ countable, choose
disjoint sets $\{R_0,R_1\}$ from $M$ such that $R_i=\{d^i_n:n\in\omega\}$ is a set of $E_i$-representatives 
for $i=0,1$ and put $D^i_n:=\{d^i_n\}$ for every $n\in\omega$.}
However, Theorem~\ref{binary} shows that a rather mild class of expansions will be  Borel complete.

\begin{Definition} \label{coloring} {\em Given a countable structure $M$, a {\em coloring of $M$} is a function $\cc:M\rightarrow\omega$.
Formally, when we write $(M,\cc)$, we are considering the expansion  of $M$ to the $\LL^*=\LL\cup\{U_n:n\in\omega\}$-structure
$M^*$, where $U_n(M^*)=\cc^{-1}(n)$ for each $n\in\omega$.  
Let $$\CC(M)=\{(M,\cc):\ \hbox{$\cc$ a coloring of $M$}\}\ \hbox{and} \ \CC(\Phi)=\{\CC(M):M\models\Phi\}$$
}
\end{Definition}

In preparation for the following theorem, we look at the class $\BP$ of all {\em bipartite graphs on $\omega\times\omega$}, 
which naturally correspond to subsets $R\subseteq \omega\times\omega$.  We say that two bipartite graphs $R, R'$ are isomorphic if there are permutations $\sigma_0, \sigma_1 \in \mbox{Sym}(\omega)$ such that for all $(n, m) \in \omega \times \omega$, we have $(n, m) \in R$ if and only if $(\sigma_0(n), \sigma_1(m)) \in R'$. 

Call a bipartite graph {\em reduced}
if, for all distinct $n\neq n'$, $\{m\in\omega: R(n,m)\}\neq \{m\in\omega: R(n',m)\}$ and dually, $\{n\in\omega:R(n,m)\}\neq\{n\in\omega:R(n,m')\}$ for distinct $m\neq m'$.  Let $\BP^*$ denote the set of reduced bipartite graphs 
$(\omega\times\omega,R)$.  It is well known that the class $(\BP^*,\cong)$ is Borel complete.

\begin{Theorem} \label{binary}   Suppose a countable structure $M$ admits cross-cutting absolutely indiscernible sets.
Then $\CC(M)$ is Borel complete.  
\end{Theorem}

\begin{proof}  Choose such an $M$ and fix $E_0,E_1,\D^0_*,\D^1_*$ witnessing this; let (1), (2), (3) refer to the items of the definition of crosscutting absolutely indiscernible sets.  We replace each $D^i_n$ by its $E_i$-saturation, i.e., replace $D^i_n$ by $\{a'\in M:M\models E_i(a',a)$ for some $a\in D^i_n\}$.  This does not disturb any of (1)--(3).  
It suffices to define a Borel reduction from $\BP^*$ into $\CC(M)$.  First, call an element $c\in M$ a {\em grid point}
if $c \in \D^0_* \cap \D^1_*$.     It follows from (1) and (2) that if $c\in M$ is a grid point,
then there are unique $n,m\in\omega$ with $c \in D^0_n \cap D^1_m$. We will always put $\cc(x)=0$ iff $x$ is not a grid point. As both $E_0,E_1$ are invariant, it follows that $\D^i_*$ will be invariant under any color-preserving isomorphism, i.e., sending grid points to grid points.

Now, suppose we are input a reduced $R\subseteq\omega\times\omega$.  We define the coloring $\cc_R:\omega\rightarrow\{0,1,2\}$ by putting $\cc_R(x)=0$ for all non-grid points, and for grid points $c \in D^0_n \cap D^1_m$, putting $\cc_R(c)=1$ if and only if $(n, m) \in R$, and put $\cc_R(c)=2$ otherwise. By the previous comments this is well-defined, since $(n, m)$ is uniquely determined by $c$.

The following Claim (along with the dual claim for $\D^1_*$) uses the fact that $R$ is reduced.


\medskip\noindent{\bf Claim.}  Choose any $a,a'\in\D^0_*$.  Then $a,a'$ are in the same $D^0_n$ if and only if
for every $b\in \D^1_*$, $\cc_R(c)=\cc_R(d)$
for every $c\in [a]_{E_0}\cap [b]_{E_1}$ and  every $d\in [a']_{E_0}\cap [b]_{E_1}$.

\begin{proof}  $(\Rightarrow)$ Assume $a,a'\in D^0_n$ and choose $b,c,d$ as above.  Say $b\in D^1_m$.  
Then $\cc_R(c)=1$ iff $R(n,m)$ iff $\cc_R(d)=1$ by definition.  

$(\Leftarrow)$  Say $a\in D^0_n$, $a\in D^0_{n'}$ with $n\neq n'$.  As $R$ is reduced (and by symmetry), choose
$m$ such that $R(n,m)$ and $\neg R(n',m)$ hold. By (2), choose $c\in [a]_{E_0}\cap [b]_{E_1}$ and 
$d\in [a']_{E_0}\cap [b]_{E_1}$.  Then $\cc_R(c)=1$, while $\cc_R(d)=2$.
\end{proof}

To show the above is a Borel reduction, choose isomorphic $(\omega\times\omega,R)\cong (\omega\times\omega,S)$ from $\BP^*$.
Choose an isomorphism $(\sigma_0,\sigma_1)$ and choose $\tau\in\Aut(M)$ by (3).  Then 
$\tau:(M,\cc_R)\mapsto (M,\cc_S)$ is an $\LL^*$ isomorphism.  
Conversely, suppose $h:(M,\cc_R)\rightarrow (M,\cc_S)$ is an $\LL^*$-isomorphism.  
Then, using the Claim, define $\pi_0:\omega\mapsto \omega$ by letting $\pi_0(n)$ be the unique $n^*$ such that
for any $a\in D^0_n$, $h(a)\in D^0_{n^*}$, and defined $\pi_1(m)$ to be the unique $m^*$ such that for any $b\in D^1_m$,
$h(b)\in D^1_{m^*}$.  As $h$ is color preserving, it follows that for any $(n,m)\in\omega\times\omega$, for any
$a\in D^0_n$, $b\in D^1_m$, and $c\in [a]_{E_0}\cap [b]_{E_1}$, 
$R(n,m)$ iff $\cc_R(c)=1$ iff  $\cc_S(h(c))=1$ iff  $S(\pi_0(n),\pi_1(m))$.  Thus $(\pi_0,\pi_1)$ is a bipartite graph isomorphism.
\end{proof}

\section{Down-Finite Posets}

\begin{Definition}
{\em 
    Given a poset $(P,\le)$ and $Q \subseteq P$, let the {\em downward closure} of $Q$, $\dc(Q)$, be the set of all $p \in P$ such that $p \leq q$ for some $q \in Q$. We say that $Q$ is {\em downward closed} if $Q = \dc(Q)$. When $Q = \{q\}$ is a singleton we write $P_{\leq q}$ instead of $\dc(q)$.  
    
    Say that $(P,\le)$ is {\em down-finite} if for all $p \in P$, $P_{\leq p}$ is finite.

    For $\alpha$ an ordinal, a {\em chain of length $\alpha$} from $P$, or just an {\em $\alpha$-chain}, is a sequence $(p_\beta: \beta < \alpha)$ from $P$ with $p_\beta < p_{\beta'}$ for all $\beta < \beta' < \alpha$. We define the {\em height} of $P$, $\hgt(P)$, to be the supremum of lengths of chains from $P$. When $(P,\le)$
    is down-finite, every chain from $P$ is of length at most $\omega$, so $\hgt(P)$ is always finite or $\omega$.

    If $\hgt(P) < \omega$ we say that $(P,\le)$ is of {\em bounded height.} This includes the case whenever $P$ is finite.
     For every $q\in P$,  the {\em height of $q$}, $\hgt_P(q)$,
     is defined to be  $\hgt(P_{\leq q})$.   If $P$ is down-finite, then $\hgt(q) < \omega$ for every $q\in P$ (it is always $\le |P_{\le q}|$).
}
\end{Definition}  

The following fact is easy.

\begin{Fact}  \label{ht}  Suppose $(P, \leq)$ is down-finite. If $q \in P$ with $\hgt(q)=n$, then for every chain $(r_1,\dots,r_n)\subseteq P_{\le q}$ we must have
$r_n=q$ and, for every $k\le n$, $\hgt(r_k)=k$.  In particular, if $\hgt(q)=n$ and $k\le n$, then some $r\le q$ has $\hgt(r)=k$.  
\end{Fact}

\begin{proof}  First note that if $r_n\neq q$, then concatenating $q$ at the end would give a chain of length $n+1$, which is forbidden.
More generally, for any $k\le n$, the chain $(r_1,\dots,r_k)$ witnesses that $\hgt(r_k)\ge k$.  However, if there were a longer chain
in $P_{\le r_k}$ we could concatenate $(r_{k+1},\dots,r_n)$ to it, contradicting $\hgt(q)=n$.
\end{proof}

\begin{Definition}\label{botDef}{\em
Let $(P,\le)$ be a down-finite poset.
Two subsets $Q,R\subseteq P$ are {\em orthogonal}, $Q\bot R$, if $\{q,r\}$ are incomparable for every $q\in Q$, $r\in R$.  

Say that $P$ is {\em narrow} if whenever $Q \bot R$ are orthogonal subsets of $P$, at least one of $Q,R$  
has bounded height. 
}
\end{Definition}

\begin{Lemma}
Suppose $P$ is narrow. Then $P$ does not admit a family of arbitrarily long pairwise orthogonal finite chains.
\end{Lemma}
\begin{proof}
    Suppose we had arbitrarily long pairwise orthogonal finite chains $(C_n: n < \omega)$, say $C_n$ is of length $n$. Let $Q = \{C_{2n}: n < \omega\}$ and let $R = \{C_{2n+1}: n < \omega\}$. Then $Q \bot R$ witnesses that $P$ is not narrow.
\end{proof}

\begin{Lemma}
    Suppose $(P, \leq)$ is down-finite, narrow, and of unbounded height. Then there is some $p \in P$ such that $P_{>p}$ has unbounded height.
\end{Lemma}
\begin{proof}
Suppose not. We construct a family of arbitrarily long pairwise orthogonal finite chains.

Suppose we have found pairwise orthogonal chains $(C_i: i < n)$ from $P$. It suffices to find an $n$-chain $C_n$ from $P$ which is orthogonal to each $C_i$, for $i < n$. Let $k$ be large enough so that for all $p \in \bigcup_{i < n} C_i$, $\hgt(P_{> p}) < k$ and $\hgt(p) \leq k$. Let $C' = (r_j: j < 2k+n)$ be a chain from $P$ of length $2k+n$.  Let $C_n = (r_j: k \leq j < k+n)$, a chain of length $n$. We claim that $C_n$ is as desired. Suppose $i < n$ and $p \in C_i$ and $k \leq j < k+n$. Then $p \not \leq r_j$, as otherwise $( r_{j+1}, \ldots, r_{2k+n-1})$ would be a chain of length at least $k$ above $p$, contradicting $\hgt(P_{> p}) < k$. Similarly, $r_j \not \leq p$, as otherwise $(r_0, r_1, \ldots, r_j)$ would be a chain of length at least $k+1$ from $P_{\leq p}$, contradicting $\hgt(p) \leq k$.
\end{proof}

\begin{Theorem}\label{omegaChain}  Suppose $(P, \leq)$ is down-finite, narrow, and of unbounded height. Then $P$ contains an $\omega$-chain.
\end{Theorem}

\begin{proof}
Let $P_*$ be the set of all $p \in P$ such that $P_{>p}$ has unbounded height. 
The preceding lemma says that $P_*$ is nonempty. We show that no 
element of $P_*$ is maximal, from which it follows that $P_*$ contains an $\omega$-chain.  

 Choose any $p \in P_*$. Then $(P_{>p},\le)$ is narrow, since $(P,\le)$ is, and is of unbounded height. Hence, by the preceding lemma, there is $q \in P_{>p}$ with $P_{>q}$ of unbounded height, hence $q \in P_*$ and $q > p$.
\end{proof}

\section{An interesting family of mutually algebraic theories}

\begin{Definition}  \label{start} {\em 

 Let $\PP$ consist of all triples $(P,\le,\delta)$ satisfying:
 \begin{itemize} 
 \item $(P,\le)$ is a countable, down-finite poset;
\item   $\delta:P\rightarrow\omega\setminus\{0,1\}$ is an arbitrary function.
\end{itemize}
Let  $\LL_P=\{E_q:q\in P\}$ with each $E_q$ a binary relation symbol, and for each $(P,\le,\delta)\in\PP$ let:
\begin{itemize}
\item  $\F(P,\le,\delta):=\prod_{q\in P} \delta(q)$; and
\item  $\MM(P,\le,\delta)$ is the $\LL_P$-structure with universe $\F(P,\le,\delta)$ where, for each $q\in P$, $E_q$ is interpreted as:
$$E_q(f,f') \quad \Longleftrightarrow  \bigwedge_{q'\le q} f(q')=f'(q')$$
\end{itemize}
Note that $\F(P,\le,\delta)$ does not depend on the partial order structure, but $\MM(P,\le,\delta)$ certainly does.
When $\le$ and $\delta$ are clear from context (which is usually the case) we will write just $\mathcal{F}_P, \MM_P$.
}

\end{Definition}

If $P$ is finite, then $\F_P$ and $\MM_P$ will be finite, and thus uninteresting.  
However, when $(P,\le,\delta)\in \PP$ and $P$ is (countably) infinite, then $|\F_P|$ has cardinality $2^{\aleph_0}$ and is a compact Polish space when endowed with
 the Tychonoff topology.
 It is easily seen that any such $\MM_P$ satisfies the following axiom schemes, which we dub $T(P,\le,\delta)$ (or just $T_P$).
 
 \begin{itemize}
 \item  Each $E_q$ is an equivalence relation;
 \item  If $q'\le q$, then $E_q$ refines $E_{q'}$ and, moreover, letting $E_{<q}(x,y):=\bigwedge_{q'<q} E_{q'}(x,y)$,
 then $E_q$ partitions each $E_{<q}$-class into precisely $\delta(q)$ classes.  [Since $\{q'\in P:q'<q\}$ is finite, $E_{<q}$ is a finite conjunction of atomic $\LL_P$-formulas.]
 \item If $Q\subseteq P$ is downward closed and finite,  then for every sequence $(a_q:q\in Q)$ satisfying
 $E_{q'}(a_{q'},a_q)$ for all $q'\le q$, there is $a^*$ such that $E_q(a^*,a_q)$ for every $q\in Q$.
 \end{itemize}
 
 It is easily checked that when $P$ is infinite, then the  axioms $T_P$  admit elimination of quantifiers in the language $\LL_P$, and hence generate the complete theory of $\MM_P$.
 As a consequence, when $P$ is infinite, then $T_P$ is mutually algebraic, and moreover, for any $M\models T_P$, $\acl(A)=A$ for every $A\subseteq M$ and $x=x$ generates a complete 1-type. Further, in any $M \models T_P$, the group of elementary permutations of $\mbox{acl}^{eq}(\emptyset)$ is topologically isomorphic to $\mbox{Aut}(\MM_P)$; in particular the latter is a compact Polish group.
 
 In describing the Borel complexity of countable models of such a $T(P,\le,\delta)$, it is useful to consider the following universal sentence $\Psi\in L_{\omega_1,\omega}$ 
 $$\forall x\forall y \left(\bigwedge_{q\in P} E_q(x,y)\rightarrow x=y\right)$$
Let $\Phi(P,\le,\delta):= T(P,\le,\delta)\cup\{\Psi\}$.  
 Visibly, $\MM(P, \leq, \delta) \models\Phi(P,\le,\delta)$.  

 We establish some initial consequences.

 \begin{Fact} \label{check}   Fix any $(P,\le,\delta)\in\PP$.
 \begin{enumerate}
 \item  A subset $M\subseteq \MM_P$ is an elementary substructure iff $M$ is a dense subset of $\F_P$.
 \item Every elementary substructure $M\preceq \MM_P$ is a model of $\Phi_P$.
 \item  Suppose $A \subseteq \MM_P$ and $A \subseteq M \models \Phi_P$. Then there is an isomorphic embedding $f: M \to \MM_P$ which is the identity on $A$. In particular, models of $\Phi_P$ have size at most $\beth_1$.
 \item If $M\preceq \MM_P$, then every automorphism $\sigma\in\Aut(M)$ extends uniquely to an automorphism of $\MM_P$. Thus $\Aut(M)$ is isomorphic to a subgroup of $\Aut(\MM_P)$.
 \end{enumerate}
 \end{Fact}
 \begin{proof}
     (1) If $M$ is not dense in $\F_P$ then there is some finite downward closed $Q \subseteq P$ such that $M$ does not meet every $\bigwedge_{q \in Q} E_q$-class. But $T_P$ entails that there are $\prod_{q \in Q} \delta(q)$ such classes, so $M \not \models T_P$. The converse is similar, noting that $M \preceq \MM_P$ if and only if $M \models T_P$, by the quantifier elimination.

     (2) Clear.

     (3) Suppose $A, M$ are given. It suffices to show that for all $b \in M$ there is $f \in \MM_P$ such that $\qftp(A, b) = \qftp(A, f)$, since then we can replace $A$ by $A \cup \{f\}$, and replace $M$ by a copy over $Af$, and continue by transfinite induction. We can also suppose of course $b \not \in A$.

     Let $Q:=\{p\in P: M\models E_p(b,a)$ for some $a\in A\}$.
     Then $Q$ is downward closed.  Choose $s\in\prod_{q\in Q} \delta(q)$ such that $s(q)=a(q)$ for some/every $a\in A$ such that $M\models E_q(b,a)$.  
     Let $R$ be the set of  minimal elements of $P \backslash Q$. For each $r \in R$, $P_{<r} \subseteq Q$. Let $A'_r = \{a \in A: M\models E_q(a, b) \mbox{ for all } q < r\}$ and let $A_r \subseteq A'_r$ be a choice of representatives for $A'_r/E_r$.
     Since $r\not\in Q$, $M\models\neg E_r(a,b)$ for every $a\in A_r$.  Since $M\models T_P$, this implies $\delta(r)\ge |A_r|+1$.  So,  choose $t\in \prod_{p\in Q\cup R} \delta(p)$
     to extend $s$, but with $t(r)\neq a(r)$ for every $r\in R$ and $a\in A_r$.   Then any $f\in\F_P$ extending $t$ works; we use that $M \models \Phi_P$ to verify that $f \not \in A$.
    





     (4) It is straightforward to check that every automorphism of $\MM_P$ is continuous, from which uniqueness follows. Further, by (3), if $\sigma: M \to M$ then we can find some extension $\tau: \MM_P \to \MM_P$, and it suffices to show that $\tau$ is surjective. For this it suffices to note that $\MM_P$ is a maximal model of $\Phi_P$, again by (3) (take $A = \MM_P$).
 \end{proof}

 We will see that as we vary $(P,\le,\delta)$ within $\PP$, we obtain radically different classes in regard to Borel complexity of the mutually algebraic theories $T(P,\le,\delta)$.  For this, it is useful to consider the class  $\CC(\Phi_P)$ of colored models $(M,\cc)$ for $M\models\Phi_P$ described in Definition~\ref{coloring}.
 The following Lemma is almost immediate.

 \begin{Lemma} \label{shift1}  For any $(P,\le,\delta)\in \PP$, the first order theory $T_P$
 is Borel equivalent to $\CC(\Phi_P)$.
 \end{Lemma}

 \begin{proof}  Given a nonempty set $X$, let $h(X)$ denote its cardinality if $X$ is finite, and otherwise $h(X) := 0$.
 Fix any $(P,\le,\delta)\in\PP$ and consider the type-definable  equivalence relation 
 $E_P(x,y):=\bigwedge_{p\in P} E_p(x,y)$.   Given any $M\models T_P$,  the quotient $M/E_P\models \Phi_P$.
 For any countable $M\models T_P$, define a coloring $\cc:M/E_P\rightarrow\omega$, where $\cc(a)=h([a]_{E_P})$
 for each $a\in M$.
 The map $M\mapsto (M/E_P,\cc)$ is clearly a Borel reduction from the class of countable models of $T_P$ to $\CC(\Phi_P)$.
 For the reverse direction, given a countable $(M,\cc)\in\CC(\Phi_P)$, we construct $M_\cc\models T_P$
 as follows:  Choose a family $\{F(a):a\in M\}$ of pairwise disjoint sets, with $h(F(a))=\cc(a)$ for each $a\in M$.
 Let $M_\cc$ be the $\LL_P$-structure with universe $\bigcup\{F(a):a\in M\}$ and, for each $p\in P$, $E_p$ is interpreted as
 $E_p(M_\cc)=\bigcup\{F(a)\times F(b): a,b\in E_p(M)\}$.
 \end{proof}

 As a preamble to what follows, a general technique for showing that $T(P,\le,\delta)$ is Borel complete will be to show that $\Phi(P,\le,\delta)$ admits cross-cutting absolutely indiscernible sets and then applying Lemma~\ref{shift1} and Theorem~\ref{binary}.  Alternatively, a way of showing non-Borel completeness (and, in fact prove that `countable sets of countable sets of reals' do not embed) of $T(P,\le,\delta)$ is to show that $\Phi(P,\le,\delta)$ forbids nested sequences and apply Theorem~\ref{topTrivialTheorem} to $\Phi(P,\le,\delta)$. 

 In the remainder of this section we discuss several examples. 
 We recall the taxonomy mentioned in the Introduction.

 \begin{Definition} \label{taxonomy} {\em Suppose $(P,\le,\delta)\in \PP$.
 \begin{itemize}  
 \item  
 $(P,\le,\delta)$ is {\em bounded} if the poset $(P,\le)$ is of bounded height and $\delta$ is bounded on $P$.  

 \item  We call a subset $Q\subseteq P$ {\em bounded} if
 $(Q,\le,\delta\mr{Q})\in\PP$ is bounded.
 
 \item  $(P,\le,\delta)$
 is {\em unbounded} if it is not bounded.

 \item  $(P,\le,\delta)$ is {\em minimally unbounded} if
 it is unbounded, but for all downward closed $Q \subseteq P$, either $Q$ or $P \backslash Q$ is bounded.

 \end{itemize}

We will also need the following. Suppose $(P, \leq, \delta) \in \PP$. Then for any $Q \subseteq P$ downward closed, let $E_Q$ be the type-definable equivalence relation $\bigwedge_{q \in Q} E_q$; this definition makes sense in any $M \models T_P$.

}
 \end{Definition}


 First, suppose $P$ is an infinite antichain, i.e., $\hgt(P)=1$. Then the theory $T_P$ is simply $CC(\delta)$, as discussed in the Introduction.
 Combining results from \cite{reducts} and \cite{expansions}, $T_P$ is Borel complete if and only if $\delta$ is unbounded. When $\delta$ is bounded, $\Phi_P$ forbids nested sequences; when $\delta$ is unbounded, then we can split $P$ into two pieces $Q \cup R$ with $\delta$ unbounded on both, and then $E_Q$ and $E_R$ witness the existence of crosscutting absolutely indiscernible sets.
 With Corollary~\ref{introCor} we see this behavior is typical of any example with $(P,\le)$ of bounded height.  
 
 
 Let $P$ be a single $\omega$-chain and let $Q =\{p, q\}$ be an antichain of size two. Then $(P\times Q,\le,\delta)$ is unbounded, but not minimally unbounded for any choice of $\delta$.
 Theorem~\ref{bot} and Lemma~\ref{shift1} show that $T_{P \times Q}$ is Borel complete (regardless of the choice of $\delta$). Indeed, $E_{P\times\{p\}}$ and $E_{P \times \{q\}}$ 
 witness that $\Phi_{P\times Q}$ admits cross-cutting absolutely indiscernible sets.

 Let $P$ be the union of $n$-chains $C_n$ for $n < \omega$ with no relations among the $n$-chains. Let $\delta$ be arbitrary. Then again, $(P,\le,\delta)$ is unbounded, but not minimally unbounded.   Theorem~\ref{bot} shows that $T_{P}$ is Borel complete (for any $\delta$). Indeed, let $Q = \{C_{2n}: n < \omega\}$ and $R = \{C_{2n+1}: n < \omega\}$; then $E_Q$ and $E_R$ witness the existence  of  cross-cutting absolutely indiscernible sets.

 Let $P$ be a single $\omega$-chain and let $Q$ be a chain of length $2$. By Corollary~\ref{FirstHalf}, $T_{P \times Q}$ is Borel complete (for any $\delta$). By themselves, equivalence relations of the form $E_R$ for $R \subseteq P$ downward closed do not witness the existence of cross-cutting absolutely indiscernible sets.

 Let $P$ be a single $\omega$-chain. Then $T_P$ is simply $REF(\delta)$ from the Introduction, and we know that $T_P$ is not Borel complete. Indeed, by Theorem~\ref{app}, 
 no reduct of any model of $REF(\delta)$ is Borel complete either.  Regardless of $\delta$, $(\omega,\le,\delta)$ is minimally unbounded.  

 Let $P = \{p_n: n < \omega\} \cup \{q_{n, m}: n, m < \omega\}$ where $\{p_n: n < \omega\}$ is an $\omega$-chain and $q_{n, m}: m < \omega$ is an antichain above $p_n$. Let $\delta$ be identically three. 
 Then $(P,\le,\delta)$ is minimally unbounded.  By Theorem~\ref{SecondHalf}, under sufficient large cardinals $T_P$ is not Borel complete. We conjecture that its potential cardinality is $\beth_2$ (and that this can be proven in $ZFC$) but at present we cannot even prove it is less than $\infty$.

\section{Characterizing when $S_\infty$ divides}

Fix $(P,\le,\delta)\in\PP$.  In this section we characterize when $S_\infty$ divides the automorphism group of some countable $M \models \Phi_P$.

 We want to compare elements $f,f'\in \MM_P$.  Let 
$$\bigwedge (f,f'):=\{q\in P:f(q')=f'(q')\ \hbox{for all $q'\le q$}\}$$
and, for $A\subseteq \MM_P$, put $\bigwedge A:=\bigcup\{\bigwedge(a,a'):a\neq a'\in A\}$.

These definitions match well with the interpretations of the $E_p$'s, and the following Facts are proved merely by unpacking the definitions.

\begin{Fact} \label{wedgefacts} Suppose $(P,\le,\delta)\in\PP.$  Then
\begin{enumerate}
\item For all $f,f'\in \MM_P$, $\bigwedge(f,f')=\{q\in P:\MM_P\models E_q(f,f')\}$;
\item  If $h:A\rightarrow B$ is a bijection, then $h$ is an $\LL_P$-isomorphism if and only
if $\bigwedge(a,a')=\bigwedge(h(a),h(a'))$ for all distinct $a\neq a'$ from $A$.
\item  As a special case, if $h:A\cup\{f\}\rightarrow A\cup\{f'\}$ satisfies $h(f)=f'$
and $h$ fixes $A$ pointwise, then $h$ is an isomorphism iff $\bigwedge (f,a)=\bigwedge(f',a)$ for every $a\in A$.
\item  If $A\cong B$ then $\bigwedge A=\bigwedge B$.
\end{enumerate}
\end{Fact}


\begin{Definition} \label{densesuitable} {\em  Suppose $(P,\le,\delta)\in\PP$.  Let $\Age(\MM_P)$ denote the set of finite substructures of $\MM_P$. Call a set  $\K\subseteq\Age(\MM_P)$  {\em dense-suitable} if:
\begin{enumerate}
\item  $\K$ is closed under substructures and isomorphisms within $\Age(\MM_P)$, i.e., if $A\in\K$ and $B\subseteq \MM_P$
is isomorphic to $A$, then $B\in\K$;
\item{\bf Extendible}  The empty structure is in $\K$ and if $A\in\K$ then there is $B\in\K$ with $B\supsetneq A$;
\item{\bf Disjoint Amalgamation}  If $A,B,C\in\K$ and $A\subseteq B$, $A\subseteq C$, then there is $B'\in \K$ such that
$B'\cong_A B$ and $B' \cap C = A$ and $B'\cup C \in \K$;
\item{\bf Density}  For all $A \in \mathcal{K}$, $\{f \in \mathcal{F}: A \cup \{f\} \in \mathcal{K}\}$ is dense in 
$\mathcal{F}$.
\end{enumerate}
}
\end{Definition}



\begin{Lemma}  \label{quote}
    If there exists some $\mathcal{K}\subseteq Age(\MM_P)$ which is dense-suitable, then there is a family
    $\{D_n:n\in\omega\}$ of countable,  absolutely indiscernible  subsets of $\MM_P$,
    each of which is dense in $\F_P$.
\end{Lemma}

\begin{proof}
    Note that if $\K\subseteq \Age(\MM_P)$ satisfies Clauses (1)--(3), then if we trivially
    expand each $A\in\K$ by a unary predicate $X$ interpreted as $A^X=A$, then
    $\K$ is {\em suitable} in the sense of Definition~A.1 of \cite{modules}.  Thus,
    by Theorem~A.2 there, there is a $\K$-limit $M$, i.e., a nested union of elements of $\K$, and an equivalence relation $E$ on $M$ with infinitely many classes, each class infinite, such that every $\pi\in\Sym(M/E)$ lifts to an automorphism of $M$.
    Here, we replicate this proof, but using Clause~(4), we dovetail additional requirements that guarantee that each $E$-class is dense in $\F_P$.
    Given such an $M$ and $E$, simply let $\{D_n:n\in\omega\}$ be the $E$-classes of $M$.
    \end{proof}


The following Proposition justifies our interest in unbounded
triples $(P,\le,\delta)\in\PP$.   
With Theorem~\ref{dichotomy} we will
see that the converse of Proposition~\ref{checkdensesuitable} holds as well.


\begin{Proposition}  \label{checkdensesuitable}
Suppose $(P,\le,\delta)\in\PP$ is unbounded, i.e., either $\delta$ is unbounded on
$P$ or $(P,\le)$ has unbounded height.  Then 
 a dense-suitable $\mathcal{K}\subseteq \Age(\MM_P)$ exists, hence there is a family
    $\{D_n:n\in\omega\}$ of countable,  absolutely indiscernible  subsets of $\MM_P$,
    each of which is dense in $\F_P$.
\end{Proposition}


\begin{proof}  
We split into cases, depending on the reason why $(P,\le,\delta)$ is unbounded.

\medskip\noindent{\bf Case 1.}  $\delta$ is unbounded on $P$.

In this case, let $\K$ consist of all finite substructures $A\subseteq \MM_P$ such that
$\delta$ is bounded on $\bigwedge A$.  For clause~(1), $\K$ is trivially closed under substructure, and closure under isomorphism follows from Fact~\ref{wedgefacts}(4).
The verification of (2) and (4) are immediate from the following Claim.

\medskip
\noindent{\bf Claim.}  For all finite $Q\subseteq P$, for all $f\in\F_P$, and for all $A\in\K$
there is $h\in \F_P$ such that $h\mr{Q}=f\mr{Q}$, $A\cup\{h\}\in\K$, and $h\not\in A$.

\begin{proof}  Choose $n\in\omega$ such that $\delta(p)<n$ for all $p\in Q\cup \bigwedge A$
and $n>|A|$.  Then choose $h\in \F_P$ such that
\begin{itemize}
    \item  $h(p)=f(p)$ for all $p\in P$ with $\delta(p)<n$; and
    \item  $h(p)\neq a(p)$ for all $a \in A$, whenever $\delta(p)\ge n$. 
    
\end{itemize}
To see that $h$ works, first note that since $\delta(p)<n$ for all $p\in Q$, $h\mr{Q}=f\mr{Q}$.  Since $\delta$ is unbounded on $P$, there is at least one $p^*$ with
$\delta(p^*)\ge n$.  Since $h(p^*)\neq a(p^*)$ for each $a\in A$ we have $h\not\in A$.
Similarly, since $h(p)\neq a(p)$ for every $p$ with $\delta(p)\ge n$, we have that
$\delta$ is bounded below $n$ on $\bigwedge(h,a)$ for every $a\in A$.  Thus, $A\cup\{h\}\in\K$.
\end{proof}

To see that Disjoint Amalgamation holds,  by induction 
and applications of  Fact~\ref{wedgefacts}(3), it suffices to prove that
for all $A\in\K$, if $B=A\cup\{f\}$ and $C=A\cup\{h\}$ are in $\K$, then there is $f'\in\F_P$
such that $\bigwedge(f,a)=\bigwedge(f',a)$ for every $a\in A$ and $C\cup\{f'\}\in\K$ and $f' \not= h$.
To accomplish this, choose $n$ such that $\delta(q)<n$ for all $q\in \bigwedge B$ and all $q\in \bigwedge C$ and $n>|C|$.  Define $f'\in\F_P$ such that
\begin{itemize}
\item  $f'(p)=f(p)$ whenever $\delta(p)\le n$; and
\item $f'(p)\neq c(p)$ for all $c\in C$ whenever $\delta(p)>n$.
\end{itemize}

To see that this $f'$ works, first we show that $f' \not= h$. Indeed, choose $p \in P$ with $\delta(p) > n$, then $f'(p) \not= h(p)$.

Choose any $a\in A$.  Towards verifying $\bigwedge (f,a)=\bigwedge(f',a)$, choose $p\in \bigwedge (f,a)$.  As $\bigwedge (f,a)$ is downward closed, $\delta(q)<n$ for all $q\le p$.  Thus $f'(q)=f(q)$ for all $q\le p$, so $p\in\bigwedge(f',a)$.  Conversely, choose $p\in\bigwedge(f',a)$. This means $f'(q)=a(q)$ for all $q\le p$.
By choice of $f'$ this implies $\delta(q)\leq n$ for all $q\le p$, so $f'(q)=f(q)$ for all $q\le p$.  Hence, $p\in\bigwedge(f,a)$ as well.

Finally, to see that $C\cup\{f'\}\in \K$, since $C\in\K$ and 
$\bigwedge (f,a)=\bigwedge (f',a)$ for all $a\in A$, it remains to show that 
 $\delta$ is bounded on $\bigwedge (f',h)$, but this is immediate from the definition of $f'$.

 \medskip

 \noindent{\bf Case 2.}  $P$ is of unbounded height.

 Here, take $\K$ to be the set of all finite $A\subseteq \MM_P$ with $\hgt(\bigwedge A)<\omega$.  Again, Clause~(1) holds by Fact~\ref{wedgefacts}(3), with (2) and (4) holding  from the following Claim.

 \medskip
\noindent{\bf Claim.}  For all finite $Q\subseteq P$, for all $f\in\F_P$, and for all $A\in\K$
there is $h\in \F_P$ such that $h\mr{Q}=f\mr{Q}$, $A\cup\{h\}\in\K$, and $h\not\in A$.

\begin{proof}  We can choose $n$ such that $\hgt(q) < n$ for all $q \in Q$ and such that $\hgt(\bigwedge A) < n$.  Fix an enumeration
$A=\{a_i:i<\ell\}$ and let $h\in\F_P$
be arbitrary satisfying 
\begin{itemize}
    \item  $h(p)=f(p)$ whenever $\hgt(p) \leq n$;
    \item  $h(p)\neq a_i(p)$ for every
    $p\in P$ with $\hgt(p)=n+i+1$.
    \end{itemize}
    (Since $(P,\le)$ is of unbounded height,
    elements of height $m$ exist for every $m\ge 1$.)  

To see that any such $h\in\F_P$ works, clearly $h\mr{Q}=f\mr{Q}$ by choice of $n$.
Also, $h\not\in A$ since for every $a_i\in A$, $h(p)\neq a_i(p)$ for any $p$ of height $n+i+1$.
To see that $A\cup\{h\}\in\K$, we need only show that $\hgt(\bigwedge(h,a_i))< n+i$ for every $i$.  So choose $p\in P$ with $\hgt(p)\ge n+i$.  By Fact~\ref{ht}(2), choose
$q\le p$ with $\hgt(q)=n+i$.  By choice of $h$, $h(q)\neq a_i(q)$, so $p\not\in \bigwedge(h,a_i)$.

Finally, we show that Disjoint Amalgamation holds for $\K$.  As before,
choose $B=A\cup\{f\}$ and $C=A\cup\{h\}$,
both in $\K$.  Choose $n\in\omega$ such that
$\hgt(\bigwedge B) < n$ and $\hgt(\bigwedge C)<n$.
Define $f'\in\F_P$ satisfying:
\begin{itemize}
    \item $f'(p)=f(p)$ whenever $\hgt(p) \leq n$; and
    \item  $f'(p)\neq h(p)$ whenever $\hgt(p) > n$.
\end{itemize}
Clearly $f' \not= h$. 

\medskip\noindent{\bf Claim 2.} For all $a\in A$, $\bigwedge (f',a)=\bigwedge(f,a)$

\begin{proof}  Fix $a\in A$ and first choose $p\in \bigwedge(f,a)$.  By choice of $n$, $\hgt(q) < n$ for every $q\le p$,
so $f'(q)=f(q)$ for every $q\le p$.
Thus $p\in\bigwedge(f',a)$ as well.
Conversely, assume $p\in\bigwedge(f',a)$.
If $\hgt(p)<n$, then as above $f'(q)=f(q)$ for all $q\le p$, so $p\in\bigwedge(f,a)$.  However, if $\hgt(p)\ge n$, then by Fact~\ref{ht}(2), choose $q\le p$ with $\hgt(q)=n$.
As $\bigwedge(f',a)$ is downward closed, $q\in\bigwedge(f',a)$.  But, as $f'(q')=f(q')$ for all $q'\le q$, we would have $q\in\bigwedge(f,a)$, contradicting our choice of $n$, proving Claim 2. 
\end{proof}

Thus, by Claim 2 and Fact~\ref{wedgefacts}(3), $A\cup\{f'\}\cong B$ over $A$ and,
in light of Claim 2, to show $C\cup\{f'\}\in\K$ we need only show that $\hgt(\bigwedge(f',h))<\omega$, but this is clear by the definition of $f'$.
\end{proof}

\end{proof}

Next, we glean consequences from the assumption that $\delta$ is  bounded.  


\begin{Lemma} \label{boundeddelta}  Suppose $(P,\le,\delta)\in\PP$ and 
and $\delta$ is bounded by $m<\omega$.
Then, for every $k< \omega$ we have:
\begin{enumerate}

\item  For all $\sigma\in \Aut(\MM_P)$, for all $q\in P$ with $\hgt(q) \leq k$, and for all $f\in\F_P$ we have $\sigma^{(m!)^k}(f)(q)=f(q)$.
\item  For every nested sequence $(f_n:n\in\omega)$ from $\F_P$ (see Definition~\ref{nested}) and for every $q\in P$ with $\hgt(q) \leq k$,
$\MM_P\models E_q(f_n,f_{n'})$ for all $n,n'\ge mk$.
\end{enumerate}
\end{Lemma}

\begin{proof}   Both of these are proved by induction on $k$; since every element has height at least one, the case $k = 0$ is trivial.  

For (1) assume this holds for all $k'<k$.  Fix any $\sigma\in \Aut(\MM_P)$.   
    Put  $\psi:=\sigma^{(m!)^{(k-1)}}$ (so $\psi=\sigma$ when $k=1$ and $\sigma^{(m!)^k}=\psi^{m!}$).  By our inductive hypothesis, $\psi(f)(r)=f(r)$ for every $f\in\F_P$ and $r\in P$ with $\hgt(r) < k$.
    It follows that $\psi^s(f)(r)=f(r)$ for any $s\ge 1$.  In particular, $\psi^{m!}(f)(r)=f(r)$ for every $f\in\F_P$ and $r\in P$ with $\hgt(r) < k$.
 It remains to prove (1) for $q\in P$ with $\hgt(q)=k$.   Choose any such $q$ and fix $f\in\F_P$.
 By pigeon-hole choose $\ell<\ell'<m$  such that $\psi^\ell(f)(q)=\psi^{\ell'}(f)(q)$. As $r<q$ implies $\hgt(r) < k$ our sentences above  give that
 $\psi^\ell(f)(r)=f(r)=\psi^{\ell'}(f)(r)$  for every $r<q$.  
 Thus, by our interpretation of $E_q$, we have $\MM_P\models E_q(\psi^\ell(f),\psi^{\ell'}(f))$.  Since $\psi^{-\ell}\in \Aut(\MM_P)$ we obtain that
$\MM_P\models E_q(\psi^t(f),f)$, where $t=\ell'-\ell$.  In particular, $f(q)=\psi^t(f)(q)$.
As $0<t<m$, $t$ divides $m!$, so $f(q)=\psi^{m!}(f)(q)$, as required.

For (2), choose any nested sequence $(f_n:n\in\omega)$ from $\F_P$.  

Suppose we have verified (2) at $k$ and choose any $q\in P$ with $\hgt(q)=k+1$.  Note that every $r<q$ has $\hgt(r) \leq k$, so $\MM_P\models E_r(f_n,f_{n'})$ for all $n,n'\ge km$.
By pigeon-hole, choose $\ell,\ell'$ with $km\le \ell<\ell'<km+m$ and $f_{\ell}(q)=f_{\ell'}(q)$.  Coupled with the sentences above, we have
$f_{\ell}(q')=f_{\ell'}(q')$ for every $q'\le q$, so by our interpretation of $E_q$, we obtain $\MM_P\models E_q(f_{\ell},f_{\ell'})$.  As $(f_n:n\in\omega)$ is nested,
we conclude that $\MM_P\models E_q(f_n,f_{n'})$ for all $n,n'\ge (k+1)m$.
\end{proof}


\begin{Corollary}  \label{boundedconsequence}  Suppose $(P, \leq, \delta)$ is bounded, say $\hgt(P) \leq k$ and $\delta$ is bounded by $m$. Then:
\begin{enumerate}
\item  For every $\sigma\in \Aut(\MM_P)$, $\sigma^{(m!)^k}=id$, hence $\Aut(\MM_P)$ has bounded exponent.
\item  Every nested sequence $(f_n:n\in\omega)$ from $\MM_P$ is eventually constant, hence $\Phi_P$ forbids nested sequences (using Fact ~\ref{check}(3)).
\end{enumerate}
\end{Corollary}

We culminate our previous results of this section into the following theorem.  We will use this theorem both when $P$ is bounded, and more generally, when we relativize to bounded subsets of $P$.

\begin{Theorem}   \label{dichotomy}  The following are equivalent for any $(P,\le,\delta)\in\PP$.
\begin{enumerate}
\item  $(P,\le,\delta)$ is unbounded;
\item  $\Phi_P$ admits absolutely indiscernible sets $\{D_n:n\in\omega\}$
with  each $D_n$ dense in $\F_P$;
\item  $\Phi_P$ admits absolutely indiscernible sets;
\item  $S_\infty$ divides $\Aut(M)$ for some countable $M\models \Phi_P$;
\item  $\Aut(\MM_P)$ has unbounded exponent;
\item  $\Phi_P$ admits nested sequences.
\end{enumerate}
\end{Theorem}

\begin{proof}
$(1)\Rightarrow(2)$ is by Proposition~\ref{checkdensesuitable}.

$(2)\Rightarrow(3):$ is immediate.   

$(3)\Rightarrow(4):$ this is Fact~\ref{Allisonthm}.

$(4)\Rightarrow(5)$: suppose $S_\infty$ divides $\Aut(M)$. Then $\Aut(M)$ is of unbounded exponent. By Fact~\ref{check}(4), $\Aut(M)$ is isomorphic to a subgroup of $\Aut(\MM_P)$, hence the latter has unbounded exponent as well.

$(5)\Rightarrow(1)$  and $(6)\Rightarrow(1)$ follow directly from the two parts of Corollary~\ref{boundedconsequence}.


$(4) \Rightarrow (6)$: Suppose (6) fails. Then every expansion of $\Phi_P$ forbits nested sequences, so by Theorem \ref{topTrivialTheorem}, no expansion of $\Phi_P$ is Borel complete, so (4) fails by Theorem~5.5 of \cite{expansions}. 
\end{proof}

\section{Quotients and Substructures}\label{QuotientsSection}
    Suppose $(P,\le,\delta)\in\PP$
and choose a non-empty subset $Q\subseteq P$.  Then $(Q,\le,\delta\mr{Q})\in \PP$, so we can apply all of the preceding discussion to this triple.  So recall that $\F_Q=\prod_{q\in Q} \delta(q)$ and $\MM_Q$
is the $\LL_Q$-structure with universe $\F_Q$, where
$$\MM_Q \models E_q(f,g)\quad \hbox{iff}\quad \bigwedge_{{q'\le q, q'\in Q}} f(q')=g(q')$$

    \begin{Definition}{\em 
        Suppose $(P, \leq, \delta) \in \PP$ and $R \subseteq P$ is downward closed. Then let $E_R$ be the type-definable equivalence relation $\bigwedge_{p \in R} E_p$. Given $M \models \Phi_P$ let $[M]_R$ denote the set of $E_R$-equivalence classes $\{[f]_R: f \in M\}$. $[M]_R$ is naturally an $\LL_R$-structure, where we put $E_p([a]_{R} ,[b]_R)$ if and only if $E_p(a, b)$; this is well-defined because $R$ is downward closed.

        In the case when $M \preceq \MM_P$ let $M \mr{R} := \{a \mr{R}: a \in M\}$, a subset of $\MM_R$; we thus view $M \mr{R}$ as an $\LL_R$-structure.
        }
    \end{Definition}

    Clearly $[M]_R$ and $M \mr{R}$ are isomorphic, via the map sending $[a]_R$ to $a \mr{R}$. 

\begin{Lemma}
Suppose $(P, \leq, \delta) \in \PP$ and $M \models \Phi_P$ and $R \subseteq P$ is downward closed. Then $[M]_R \models \Phi_R$. 
\end{Lemma}
\begin{proof}
    It suffices to consider the case $M \preceq \MM_P$, and then show that $M\mr{R} \models \Phi_R$, i.e. is dense in $\MM_R$. But this is clear, since $M$ is dense in $\MM_P$.
\end{proof}

    \begin{Lemma}  \label{welldefined}  Suppose $(P,\le,\delta)\in\PP$ and $R\subseteq P$ downward closed.  Let $M, N\models \Phi_P$ be arbitrary. Then any $\LL_P$-embedding $f: M \to N$ induces an $\LL_R$-embedding $[f]_R: [M]_R \to [N]_R$. When $f$ is an isomorphism so is $[f]_R$.
    \end{Lemma}
    It follows that when $M, N \preceq \MM_P$ we get a corresponding map $f': M \mr{R} \to N \mr{R}$, which is an isomorphism if $f$ is.

    \begin{proof} Since $E_R$ is type-definable, hence invariant, $f$ must induce an injection $[f]_R$ on $E_R$-classes. The definition of the structure on $[M]_R$ shows that this $[f]_R$ must be an $\LL_R$-embedding. When $f$ is an isomorphism, $[f^{-1}]_R$ is an inverse to $[f]_R$, so $[f]_R$ is an isomorphism.
    \end{proof}

We now study the structure on individual $E_R$-classes. For this a definition is convenient:

\begin{Definition}{\em
Suppose $(P, \leq, \delta) \in \PP$ and $R \subseteq P$ is downward closed. Then let $\Phi^{\forall}_{PR}$ denote the universal sentence of $\mathcal{L}_{\omega_1 \omega}$ in the language $\LL_P$ asserting of its putative model $M$:

\begin{itemize}
    \item Each $E_q$ is an equivalence relation;
    \item If $q' \leq q$, then $E_q$ refines $E_{q'}$ and, moreover, letting $E_{<q}(x, y) = \bigwedge_{q' < q} E_{q'}(x, y)$, then $E_q$ partitions each $E_{<q}$-class into at most $\delta(q)$ classes;
    \item For all $a, b \in M$, we have that $E_R(a, b)$ holds;
    \item For all $a, b \in M$, if $E_P(a, b)$ then $a = b$.
\end{itemize} }
\end{Definition}

Clearly, if $M \models \Phi_P$ and $\alpha$ is an $E_R$-class, viewed as a substructure of $M$, then $\alpha \models \Phi^{\forall}_{PR}$. We shall need the following fact:

\begin{Lemma}\label{ClassesForbidNested}
Suppose $(P, \leq, \delta) \in \PP$ and $R \subseteq P$ is downward closed. Suppose $P \backslash R$ is bounded. Then $\Phi^{\forall}_{PR}$ forbids nested sequences.
\end{Lemma}
\begin{proof}
Let $\MM$ be the unique expansion of $\MM_{P \backslash R}$ to a model of $\Phi^{\forall}_{PR}$, namely let $E_p(a, b)$ always hold for $p \in R$. By Theorem~\ref{dichotomy}, $\MM_{P \backslash R}$ forbids nested sequences, hence so does $\MM$. Thus, it suffices to show that every model of $\Phi^{\forall}_{P R}$ embeds isomorphically into $\MM$. This is like Fact~\ref{check}(3); we only used there that $M$ was a model of the universal part of $\Phi_P$.
\end{proof}
\section{Subdivisions}
The following facts are immediate.  
\begin{Fact} \label{uselater} Suppose $(P,\le,\delta)\in\PP$ and $Q\subseteq P$ is non-empty.  Then for any $f,g\in\F_P$ and
$q\in Q$ we have:
\begin{enumerate}
    \item  If $\MM_P\models E_q(f,g)$
    then $\MM_Q\models E_q(f\mr{Q},g\mr{Q})$;
    \item  If $Q$ is downward closed then $\MM_Q\models E_q(f\mr{Q},g\mr{Q})$ implies $\MM_P\models E_q(f,g)$.
\end{enumerate}
\end{Fact}

It is noteworthy that if $Q\subseteq P$ is not downward closed, then Fact~\ref{uselater}(2) can fail.  The Lemma below is a partial remedy. If $Q \subseteq P$ then for each $p \in P$ let $Q_{\leq p} = P_{\leq p} \cap Q$.

\begin{Lemma} \label{shift2}  Suppose $(P,\le,\delta)\in\PP$ and $\{Q_i:i<n\}$ be a partition of $P$.  Then for all $f,g\in \F_P$ and all $p\in P$,
$\MM_P\models E_p(f,g)$ if and only if  $\MM_{Q_i}\models E_q(f\mr{Q_i},g\mr{Q_i})$ for all $i<n$ and   $q\in Q_{\le p}^i$.
\end{Lemma}

\begin{proof}   $\MM_P\models E_p(f,g)$ iff $f(q)=g(q)$ for all $q\in P_{\le p}$ iff $f(q)=g(q)$ for all $i<n$ and all $q\in Q^i_{\le p}$ iff
$\MM_{Q_i}\models E_q(f\mr{Q_i},g\mr{Q_i})$
for all $i<n$ and all $q\in Q_{\le p}^i$.
\end{proof}


\begin{Lemma}\label{AutomorphismPatching}
    Suppose  $(P,\le,\delta)\in\PP$ and suppose 
    $\{Q_i: i < n\}$ is a partition of $P$ into pieces, and each $i < n$,  suppose $M_{Q_i}\preceq \MM_{Q_i}$.
    Then, letting $M:=\{f\in\F:f\mr{Q_i}\in M_{Q_i}$ for all $i<n\}$, we have $M\preceq \MM_P$ and, for any choice of automorphisms 
     $\sigma_i\in \Aut(M_{Q_i})$, the map $\tau:M\rightarrow M$
     defined as $\tau(f) = \bigcup_{i < n} \sigma_i(f \mr{Q_i})$ is an automorphism of $M$.
\end{Lemma}
\begin{proof}
    $\tau$ is clearly a bijection, so it suffices to show $\tau$ preserves 
    $E_p$ for every $p\in P$.  Given $p$, choose any $f,g\in M$.  Since each $\sigma_i\in\Aut(M_{Q_i})$ for every $i<n$ we have
    $$M_{Q_i}\models E_q(f\mr{Q_i},g\mr{Q_i}) \quad \leftrightarrow\quad 
    M_{Q_i}\models E_q(\sigma_i(f\mr{Q_i}),\sigma_i(g\mr{Q_i}))$$
    for every $q\in Q^i_{\le p}$.  However, for each $i<n$, $\tau(f)\mr{Q_i}=\sigma_i(f\mr{Q_i})$ and $\tau(g)\mr{Q_i}=\sigma_i(g\mr{Q_i})$, so
    we conclude that
    $$M\models E_p(f,g) \quad\leftrightarrow \quad M\models E_p(\tau(f),\tau(g))$$
    by two applications of Lemma~\ref{shift2}.
    \end{proof}

\begin{Theorem}  \label{bot}
If $(P,\le,\delta)\in\PP$, $Q_0,Q_1\subseteq P$ are orthogonal (see Definition~\ref{botDef}), and both $Q_0$, $Q_1$ are unbounded, then
$\CC(M)$ is Borel complete for some countable $M\models\Phi_P$, hence $\mathcal{C}(\Phi_P)$ is Borel complete.
\end{Theorem}

\begin{proof}

By Theorem~\ref{dichotomy}, for $i=0,1$ choose dense, absolutely indiscernible countable sets
$(D^i_n:n\in\omega)$ for $\MM_{Q_i}$.
By Fact~\ref{check}(1), $\D^i_*=\bigcup\{D^i_n:n\in\omega\}$  is the universe of a countable model $M_{Q_i}\preceq \MM_{Q_i}$.  Let $R=P\setminus (Q_0\cup Q_1)$ and let
$\D_R\subseteq \F_R$ be dense and contain $\overline{0}$,
the identically zero sequence on $R$.

Let $M\preceq \MM_P$ be the countable $\LL_P$-structure with
universe $\{f\in\F_P:f\mr{Q_0}\in\D^0_*$, $f\mr{Q_1}\in \D^1_*$, and $f\mr{R}\in\D_R\}$. 

For each $i=0,1$, let $E^i(x,y):=E_{\dc(Q_i)} = \bigwedge_{q\in\dc(Q_i)} E_q(x,y)$.  
For each $i = 0,1$ and each $n\in\omega$, let 
$$\tilde{D}^i_n=\{f\in\F:f\mr{Q_i}\in D^i_n\ \hbox{and}\ f\mr{R}=\overline{0}\}$$
We argue that the sets $(\tilde{D}^0_n:n\in\omega)$ and $(\tilde{D}^1_n:n\in\omega)$ are cross-cutting absolutely indiscernible sets
of subsets of $M$ with respect to $E^0,E^1$.   Being infinitarily definable,
the 
equivalence relations $E^0,E^1$ are $\Aut(M)$-invariant, so we check clauses (1)--(3) from Definition~\ref{cross-cutting}.  
For (1), given $a_0\in \tilde{D}^0_n$ and $a_1\in\tilde{D}^1_m$ and using the fact that $Q_0\bot Q_1$ implies $\dc(Q_0)\cap Q_1=\dc(Q_1)\cap Q_0=\emptyset$, 
and the fact that $a_0\mr{R}=a_1\mr{R}=\{\overline{0}\}$, 
choose $f\in M$ satisfying $f\mr{\dc(Q_0)}=a_0\mr{\dc(Q_0)}$ and 
$f\mr{\dc(Q_1)}=a_1\mr{\dc(Q_1)}$.  Then $E^0(f,a_0)$ and $E^1(f,a_1)$, as required.  For (2), suppose $a\in \tilde{D}^i_n$ and $a' \in \tilde{D}^i_*$ and $E^0(a,a')$.  Then
$a\mr{Q_i}\in D^i_n$ and $a'\mr{Q_i}=a\mr{Q_i}$, hence $a'\in \tilde{D}^i_n$ as well.  Finally, for (3), choose any $\pi_0,\pi_1\in \Sym(\omega)$.  As $\{D^i_n\}$
is a family of absolutely indiscernible sets, for $i=0,1$ 
choose $\sigma_i\in \Aut(M_{Q_i})$
such that $\sigma_i[D^i_n]=D^i_{\pi_i(n)}$ for every $n\in\omega$.
Now define 
$\tau:M\rightarrow M$  via $\tau(f) = \sigma_0(f \mr{Q_0}) \cup \sigma_1(f \mr{Q_1}) \cup f \mr{R}$. 
By Lemma~\ref{AutomorphismPatching}, $\tau\in\Aut(M)$ and $\tau[\tilde{D}^i_n]=\tilde{D}^i_{\pi_i(n)}$ for $i=0,1$, $n\in\omega$.
\end{proof}

\begin{Definition}{\em 
A subset $A \subseteq P$ is a (comparability) {\em antichain} if every pair of distinct elements from $A$ are incomparable.}
\end{Definition}

\begin{Corollary}\label{botCor}  Let $(P,\le,\delta)\in\PP$.

\begin{enumerate} 
\item  If $P$ is not narrow then some $M\models\Phi_P$ has $\C(M)$ is Borel complete.
\item  If $\delta$ is unbounded on some antichain $A\subseteq P$, then
some $M\models\Phi_P$ has $\C(M)$ Borel complete.
\end{enumerate}

\end{Corollary}

\begin{proof} (1)  Suppose $Q \bot R$ witnesses that $P$ is not narrow. As $Q$ and $R$ are not bounded,  Theorem~\ref{bot} applies.

(2)  Choose a subsequence $(p_n:n\in\omega)$ from $A$ such that $\delta(p_{n+1})>\delta(p_n)$ for each $n$.  Put $Q_0=\{p_{2n}:n\in\omega\}$
and $Q_1:=\{p_{2n+1}:n\in\omega\}$.  Then again, 
$Q_0\bot Q_1$ and each $Q_i$ is unbounded, so Theorem~\ref{bot} applies.
\end{proof}

Thus, when trying to classify the Borel complexity of $T(P, \leq, \delta)$ it is enough to restrict attention to the case where $P$ is narrow and $\delta$ is bounded on every antichain. Note then:

\begin{Lemma}\label{Ramsey}
    Suppose $(P, \leq, \delta) \in \PP$ and $\delta$ is bounded on every antichain. Then $\delta$ is bounded on $Q \subseteq P$ whenever $Q$ is of bounded height. Hence, given $Q \subseteq P$, we have that $Q$ is bounded if and only if $Q$ is of bounded height.
\end{Lemma}
\begin{proof}
    The second claim follows, so it suffices to show $\delta$ is bounded on $Q$ whenever $Q$ is of bounded height. Suppose towards a contradiction it were unbounded, say $\delta(p_n) \geq n$ with $p_n \in Q$. After applying Ramsey's theorem we can suppose $(p_n: n < \omega)$ is an ascending chain, descending chain or antichain. Ascending chain is impossibe because $Q$ has bounded height; descending chain is impossible because $Q$ is down-finite; and antichain is impossible by hypothesis.
\end{proof}


The following Corollary generalizes the results on cross-cutting equivalence relations in \cite{reducts} and \cite{expansions}.

\begin{Corollary} \label{introCor} Suppose $(P,\le,\delta)\in \PP$ and $(P,\le)$ is of bounded height.  Then $T_P$ is Borel complete if and only if $\delta$ is bounded.
\end{Corollary}
\begin{proof}  If $\delta$ is bounded, then $(P,\le,\delta)$ is bounded, so `countable sets of countable sets of reals' do not Borel embed into $T_P$ by Theorem~\ref{topTrivialTheorem}.  By contrast, if $\delta$ is unbounded, then by Lemma~\ref{Ramsey}, $\delta$ is unbounded on some antichain, hence $\CC(M)$ is Borel complete for some countable $M\models \Phi_P$ by Corollary~\ref{botCor}.  Thus, $T_P$ is Borel complete by Lemma~\ref{shift1}.
\end{proof}

\section{Tame expansions}

\begin{Definition}  \label{tamedef} {\em Suppose $(P,\le,\delta)\in\P$ and $M\models T_P$.  
Fix any $p\in P$.  A subset $S\subseteq M^n$ is {\em $E_p$-invariant} if, for all $(a_0,\dots,a_{n-1}),(b_0,\dots,b_{n-1})\in M^n$, if $M\models \bigwedge_{i<n}E_p(a_i,b_i)$, then $[\abar\in S\leftrightarrow \bbar\in S]$.  
A subset $S\subseteq M^n$ is {\em tame} if $S$ is $E_p$-invariant for some $p\in P$.

Let $\LL^+=\LL_P\cup\{S_i(\xbar_i):i\in I\}$ where each $S_i$ is $n_i$-ary.  An expansion $M^+$ of $M \models T_P$  is a {\em tame expansion} if
the interpretation of every $S_i(\xbar_i)$ is a tame subset of $M^{n_i}$. A tame expansion of $T_P$ is the complete theory of a tame expansion $M^+$ of some $M \models T_P$.
}
\end{Definition}

As there are only finitely many $E_p$-classes, any tame $S\subseteq M^n$ is a finite union of `$E_p$-boxes', i.e., subsets of $M^n$ described by an $n$-tuple
$\alphabar=(\alpha_1,\dots,\alpha_n)$ of $E_p$-classes.
It follows that any tame expansion $M^+$ of any model $M\preceq \MM_P$ is mutually algebraic and, when $P$ is infinite, satisfies $\acl(X)=X$ for every $X\subseteq M^+$.
However, $Th(M^+)$ need not admit elimination of quantifiers in the language $\LL^+$.

\begin{Fact}\label{tameFacts}
Suppose $(P,\le,\delta) \in \PP$.
\begin{enumerate}
    \item For all $M \preceq N \models T_P$ and for all tame expansions $M^+$ of $M$, there is a unique tame expansion $N^+$ of $N$ with $M^+ \subseteq N^+$.
    \item If $M \preceq N \models T_P$ and $M^+ \subseteq N^+$ are tame expansions then $M^+ \preceq N^+$.
    \item Suppose $T$ is a tame expansion of $T_P$, with $P$ infinite. Then $T$ is Borel equivalent to $\CC(\Phi_P \land T)$. (If $P$ is finite then $T$ has exactly one countable model up to isomorphism.)
\end{enumerate}
\end{Fact}
\begin{proof}
    (1) is straightforward;  
    for (2), argue that reducts to finite languages are isomorphic; and 
    (3) is exactly like Lemma~\ref{shift1}, using (1) and (2).
\end{proof}
\begin{Theorem} \label{tametheorem} For any unbounded $(P,\le,\delta)\in\PP$, there is a countable $M\preceq \MM_P$ with a tame expansion $M^+$ for which $\CC(M^+)$ is Borel complete.

\end{Theorem}
\begin{proof}Many of the cases have been handled earlier. If $\delta$ is unbounded on an antichain, or if $P$ is not narrow, then Corollary~\ref{botCor} gives the existence of some countable $M \preceq \MM_P$ with $\CC(M)$ Borel complete, and we are done, so suppose this is not the case. Then $(P,\le)$ must be of unbounded height by Lemma~\ref{Ramsey}. 

Then by Theorem~\ref{omegaChain}, $P$ must admit an $\omega$-chain; fix one and call it $(p_n: n < \omega)$ and put
$Q_0:=\{p_{2k}:k\in\omega\}$, $Q_1:=\{p_{2k+1}:k\in\omega\}$ and $R:=P\setminus (Q_0\cup Q_1)$.  As each $Q_i$ is unbounded,
choose a dense family $\{D^i_n:n\in\omega\}$ of countable, absolutely indiscernible sets from $\F_{Q_i}$.   Choose an arbitrary countable dense subset $D_R\subseteq\F_R$ with a distinguished element $\overline{0}$
 and let $M\preceq \MM_P$ be countable with universe $\{f\in\F_P: f\mr{Q_i}\in \bigcup\{D^i_n:n\in\omega\}$ and $f\mr{R}\in D_R\}$.

Next, we describe a tame expansion $M^+$ of $M$. Let $\LL^+=\LL_P\cup\{F_n^0,F_n^1:n\in\omega\}$ and let $M^+$ be the expansion of $M$ to an $\LL^+$-structure formed by interpreting each $F^0_n(x,y)$ as $\{(f,g)\in M^2:
f(p_{2k})=g(p_{2k})$ for all $k$ such that $2k\le n\}$
and $F^1_n(x,y)$ as $\{(f,g)\in M^2: f(p_{2k+1})=g(p_{2k+1})$ for all $k$ such that $2k+1\le n\}$.
Clearly, $M^+$ is a tame expansion of $M$.  Note also that for any $f,g\in M$, 
$$M^+\models F^0_n(f,g)\quad \leftrightarrow\quad M_{Q_0}\models \bigwedge_{2k\le n} E_{p_{2k}}(f\mr{Q_0},g\mr{Q_0})$$ with an analogous statement for $F^1_n$.  

Additionally, for $i=0,1$, let $F^i_\omega:=\bigwedge_{n\in\omega} F^i_n$.  As each $F^i_n$ is atomically definable in $M^+$,  $F^i_\omega$ is $\Aut(M^+)$-invariant.  Also, the note above implies that for all $f,g\in M$, $$M^+\models F^i_\omega(f,g) \quad\leftrightarrow\quad M_{Q_i}\models \bigwedge_{q\in Q_i} E_q(f\mr{Q_i},g\mr{Q_i}) \quad \leftrightarrow \quad f \mr{Q_i} = g \mr{Q_i}$$

For $i=0,1$ and $n\in\omega$, put $\tilde{D}^i_n:=\{f\in M:f\mr{Q_i}\in D^i_n$ and $f\mr{R}=\overline{0}\}$.  We argue that these families of sets are cross-cutting with respect to $F^0_\omega,F^1_\omega$.  To verify (1), given $a_0\in \tilde{D}^0_n$
and $a_1\in \tilde{D}^1_m$, let $f\in M$ be the unique element satisfying $f\mr{Q_0}=a_0$, $f\mr{Q_1}=a_1$,
and $f\mr{R}=\overline{0}$.  Then $M^+\models F^i_\omega(f,a_i)$ by the note above.

For (2), choose any $f\in\tilde{D}^i_n$ and $g\in\tilde{D}^i_m$ with $M^+\models F^i_\omega(f,g)$.
By the definition of $\tilde{D}^i_n$ and $\tilde{D}^i_m$, $f\mr{Q_i}\in D^i_n$ and $g\mr{Q_i}\in D^i_m$.
But, by the note above we also have $f\mr{Q_i}=g\mr{Q_i}$.  Thus, $m=n$, as required. 

For (3), it suffices to show that any pair of automorphisms $\sigma_0\in\Aut(M_{Q_0})$, $\sigma_1\in\Aut(M_{Q_1})$ lift to an automorphism $\tau\in \Aut(M^+)$.   As in Lemma~\ref{AutomorphismPatching}, define $\tau:M\rightarrow M$ by $\tau(f)\mr{Q_i}=\sigma_i(f\mr{Q_i})$
for $i=0,1$ and $\tau(f)\mr{R}=f\mr{R}$.  Lemma~\ref{AutomorphismPatching} gives that $\tau\in\Aut(M)$, but we need to show that every $F^i_n$ is preserved as well.  But this follows easily from the characterization of
$M^+\models F^i_n(f,g)$ given above.
\end{proof}

\begin{Corollary}\label{dichotomyPt2}
Suppose $P \in \PP$. Then $P$ is bounded if and only if every tame expansion of $T_P$ is not Borel complete.
\end{Corollary}
\begin{proof}
If $P$ is unbounded then this is by the preceding theorem and Fact~\ref{tameFacts}(3). 
If $P$ is finite then this is obvious.
If $P$ is bounded and infinite then by Fact~\ref{tameFacts}(3) any tame expansion $T$ is Borel reducible to $\CC(\Phi_P \wedge T)$. The latter forbids nested sequences, since $\Phi_P$ does, so it has potential cardinality at most $\beth_2$. Hence, 
$|\CSS_\ptl(T)|\le \beth_2$, so  $T$ is not Borel complete.
\end{proof}

Pushing onwards:

\begin{Theorem} \label{twolevel}
Suppose $(P,\le,\delta)\in\PP$, $R\subseteq P$ is downward closed and unbounded, and $Q=P\setminus R$.  Suppose that $M^+$ is a tame expansion of some countable $M_Q\preceq \MM_Q$.  Then there is some countable $N\preceq \MM_P$ and a Borel reduction of
$\CC(M^+)$ to $\CC(N)$.
\end{Theorem}

\begin{proof}  Since $R$ is unbounded, apply Theorem~\ref{topTrivialTheorem} to get
a family $\{D_n:n\in\omega\}$ of countable, dense, absolutely indiscernible subsets from some $M_R \preceq \MM_R$.  We can suppose $M_R = \bigcup_n D_n$. 
Let $N:=\{f\in\F:f\mr{R}\in M_R$, $f\mr{Q}\in M^+\}$.
In what follows, we identify $f\in N$ with the pair $(a,b)$, where $a=f\mr{Q}\in M^+$ and $b=f\mr{R}\in M_R$.

The one-line statement of the proof that follows is that we construct a coloring of $N$ that both extends the coloring of $M^+$ and also encodes the interpretations of each (tame) $S_i\in \LL^+\setminus \LL_Q$.  To do this, we set some notation.  Write $\LL^+\setminus \LL_Q=\{S_i:i\in I\}$ and, for each $i$, let $n_i$ be the arity of $S_i$ and choose $q_i\in Q$ such that $S_i$ is $E_{q_i}$-invariant in $M^+$. First, by replacing each $S_i$ by a boolean combination, we may assume that each $S_i$ is $E_{q_i}$-irreflexive,
i.e., $M^+\models S_i(\xbar)\vdash \bigwedge_{j\neq k} \lnot E_{q_i}(x_j, x_k)$.

For each $q\in Q$, choose a (finite) set $R_q\subseteq M^+$ of representatives of the $E_q$-classes 
in $M^+$.  
For each $i\in I$,  let
$$\Pos(S_i)=\{\rbar\in R_{q_i}^{n_i}:M^+\models S_i(\rbar)\}$$
As $R_{q_i}$ is finite, so is $\Pos(S_i)$,
and since each $S_i$ is $E_{q_i}$-irreflexive,
every $\rbar\in \Pos(S_i)$ is an $n_i$-tuple of distinct elements.  
Let $$J=\{*\}\cup\{(i,\rbar): i\in I, \rbar\in \Pos(S_i)\}$$
Let $\{D_j:j\in J\}$ be a relabelling of the family of absolutely indiscernible sets $(D_n: n < \omega)$ described above.  
Choose a partition $\omega=X\sqcup Y$ 
into two infinite sets and choose disjoint sets $Y_{i}\subseteq Y$ for $i \in I$,
each of size $n_i+1$.
Enumerate $Y_i = \{m_{i, k}: k \leq n_i\}$. Let $s: \omega \to X$ be a bijection.

Before defining our Borel reduction, we introduce a partial coloring on a subset of $N$.  For all $f=(a,b)\in N$ with $b\in D_{i,\rbar}$, put
$$c^*(a,b)=\begin{cases}  m_{i,k} & \text{if $M^+\models E_{q_i}(a,(\rbar)_k)$} \\
m_{i,n_i} & \text{if $M^+\models \bigwedge_{k<n_i} \neg E_{q_i}(a,(\rbar)_k)$}\end{cases}$$

Note that since $M^+\models S_i(\rbar)$ and $S_i$ is $E_{q_i}$-irreflexive we cannot have
$a$ $E_{q_i}$-equivalent to $(\rbar)_k$ and $(\rbar)_{k'}$ for distinct $k,k'$.  

Now define $F:\CC(M^+)\rightarrow \CC(N)$
as $F(M^+,c) = (M^+, F(c))$ where $F(c)$ is the coloring on $N$ given by:
$$F(c)(f)=F(c)(a,b)=\begin{cases}
c^*(a,b), & \text{if $b\not\in D_*$} \\
s(c(a)), & \text{if $b\in D_*$} \end{cases}$$

The mapping $F$ is Borel, so to see it is a reduction, first choose $(M^+, c_1)$ and $(M^+, c_2)$ in $\CC(M^+)$ and choose an $\LL^+(c)$-isomorphism $h:(M^+,c_1)\rightarrow (M^+,c_2)$.

We construct an $\LL_P(c)$-isomorphism
$\tau:(N,F(c_1))\rightarrow (N,F(c_2))$ as follows:
For each $i\in I$ and $\rbar\in\Pos(S_i)$, there is a unique
$\rbar'\in\Pos(S_i)$ with $M^+\models E_{q_i}(h(\rbar)_k, (\rbar')_k)$ for each $k<n_i$.
Let $\hhat:\Pos(S_i)\rightarrow \Pos(S_i)$ denote this map. $\hhat$ is a bijection, since $h$ permutes the $E_{q_i}$-classes.

Let $\pi\in\Sym(J)$ be defined by $\pi(*)=*$ and  $\pi(i,\rbar)=(i,\hhat(\rbar))$.
As $\{D_j:j\in J\}$ are a family of absolutely indiscernible sets, choose $\sigma\in \Aut(M_R)$ such that $\sigma[D_j]=D_{\pi(j)}$ for each $j\in J$.  Finally, let
$\tau:N\rightarrow N$ be defined as $\tau\mr{Q}=h$ and $\tau\mr{R}=\sigma$.  
Then $\tau\in\Aut(N)$ by Lemma~\ref{AutomorphismPatching}.  To see that $\tau$ also preserves colors, there are two cases.  On one hand,
if $f=(a,b)$ with $b\not\in D_*$ then
$F(c_1)(a,b)=c^*(a,b)$ and, by our choice of $\sigma$, $F(c_2)(h(a),\sigma(b))=c^*(h(a),\sigma(b))$.
If $b\in D_{i,\rbar}$ then $\sigma(b)\in D_{i,\hhat(\rbar)}$.  So for any $k<n_i$,
$c^*(a,b)=m_{i,k}$ iff $M^+\models E_{q_i}(a,(\rbar)_k))$ and
$c^*(h(a),\sigma(b))=m_{i,k}$ iff
$M^+ \models E_{q_i}(h(a),(\hhat(\rbar))_k)$, which is the case iff $M^+ \models E_{q_i}(h(a), (h(\rbar))_k)$.  Since $h$ preserves $E_{q_i}$, we get that $F(c_1)(f) = F(c_2)(\tau(f))$ whenever $f\mr{R}\not\in D_*$.
On the other hand, if $f=(a,b)$ and $b\in D_*$, then $\sigma(b)\in D_*$ as well.  Thus, $F(c_1)(f)=s(c_1)(a)$ and
$F(c_2)(\tau(f))=s(c_2(h(a))$.  But as $h$ preserves colors, we have $c_1(a)=c_2(h(a))$ so we are done in this case as well.

The converse is more interesting. We need to show that the extra relations $S_i$ are recoverable from $(N,\cc)$.  
We prove a preparatory claim. Recall that $E_R$ is the conjunction $\bigwedge_{p \in R} E_p$, hence is $\LL_P$-invariant.  
In our notation,  $(a, b) E_R (a', b')$ if and only if $b = b'$.

Fix any $i\in I$ and let  $X_i$ be the set of all $(f_j: j < n_i) \in N^{n_i}$ all from the same $E_R$-class, such that, writing $f_j = (a_j, b)$, we have that $M^+ \models S_i(\overline{a})$. 

\medskip
\noindent \textbf{Claim.} Suppose $i \in I$, and $(f_j: j < n_i) \in N^{n_i}$ is a tuple from a single $E_R$-class, say $f_j = (a_j, b)$. Then the following are equivalent:

\begin{enumerate}
    \item $\overline{f} \in X_i$;
    \item There exist $(g_j: j < n_i)\in N^{n_i}$ from a single $E_R$-class, such that for each $j$, $N \models E_{q_i}(f_j, g_j)$, and such that each $c^*(g_j)$ is defined and equal to $m_{i, j}$.
\end{enumerate}
\begin{proof}
First suppose $\overline{f} \in X_i$, i.e. $M^+ \models S_i(\overline{a})$. Let $\overline{r} \in \mbox{Pos}(S_i)$ be the unique tuple with each $E_{q_i}(a_j, (\overline{r})_j)$. Since $D_{i, \overline{r}}$ is dense we can find $b' \in D_{i, \overline{r}}$ such that $E_q(b', b)$ for all $q \in R$ with $q < q_i$. Let $g_j = (a_j, b')$. These are visibly all $E_R$-related, and by the definition of $c^*$ we have each $c^*(g_j) = m_{i, j}$. By Lemma~\ref{shift2} we have that for each $j$, $N \models E_{q_i}(f_j, g_j)$, so we get the first implication.

For the reverse implication, suppose $(g_j: j < n_i)$ are given. Write $g_j = (b_j, b')$ for some fixed $b' \in M_R$ ($b'$ exists since the $g_j$'s are all $E_R$-related). Since $c^*(g_j) = m_{i, j}$, we must have $b' \in D_{i, \overline{r}}$ for some $\overline{r} \in \mbox{Pos}(S_i)$, and then we must have each $M^+ \models E_{q_i}(b_j, (\overline{r})_j)$. But since also $M^+ \models E_{q_i}(b_j, a_j)$ (since $N \models E_{q_i}(g_j, f_j)$), we get that $M^+ \models E_{q_i}(a_j, (\overline{r})_j)$. Since $S_i$ is $q_i$-invariant,  $M^+ \models S_i(\overline{a})$ as well.
\end{proof}

Now suppose $(M^+, c_1)$ and $(M^+, c_2)$ are in $\CC(M^+)$ and $h:(N,F(c_1))\cong (N,F(c_2))$ is an $\LL_P(c)$-isomorphism. We aim to find an $\LL^+(c)$-isomorphism between $(M^+, c_1)$ and $(M^+, c_2)$. Recall that $M^+$ is a tame expansion of $M_Q \preceq \MM_Q$.

By Lemma~\ref{welldefined}, $h$ induces an $\LL_R$-automorphism $h_R: M_R \cong M_R$. Now fix $b^*\in D_*$.  Define $\hhat:M^+\rightarrow M^+$ as $\hhat(a)=h((a,b^*))\mr{Q}$.
Unpacking the definitions we have
$$h((a,b^*))=(\hhat(a),h_R(b^*)) \ \hbox{for every $a\in M_Q$}$$
It is easily shown that $\hhat\in \Aut_{\LL_Q}(M_Q)$ and that $\hhat$ is color-preserving, so
it remains to show that $\hhat$ preserves every $S_i\in \LL^+\setminus \LL$. 

Suppose $(a_j: j < n_i) \in (M^+)^{n_i}$. Put $f_j := (a_j, b^*)$ and $g_j := h(f_j)\! =\! (\hhat(a_j), h_R(b^*))$. It suffices to show that $\overline{f} \in X_i$ if and only if $h(\overline{f}) \in X_i$, but this follows from the claim, since $h$ is an $\LL_P(c)$-isomorphism.
\end{proof}







Recall the definition of $P$ minimally unbounded  from 
Definition~\ref{taxonomy}.


\begin{Corollary}\label{FirstHalf}  Suppose $(P,\le,\delta)\in\PP$, and $P$ is unbounded but not minimally unbounded. Then $\C(N)$ is Borel complete for some countable $N\preceq \MM_P$, hence $T_P$ is Borel complete. 
\end{Corollary}

\begin{proof}Choose $R \subseteq P$ downward closed so that both  $R$ and $P \backslash R$ are unbounded. Let $Q = P \backslash R$.  By Theorem~\ref{tametheorem}, choose a countable $M\preceq \MM_Q$
and a tame expansion $M^+$ of $M$ with $\CC(M^+)$ Borel complete.  Then, by Theorem~\ref{twolevel} choose a countable $N\preceq \MM_P$ for which there is a Borel
reduction from $\CC(M^+)$ to $\CC(N)$.  That this implies $T_P$ is Borel complete follows from Lemma~\ref{shift1}.

\end{proof}

\section{$(P,\le,\delta)$ minimally unbounded}

We have, at this point, established Theorems~\ref{Intro1} and ~\ref{Intro2} from the Introduction. Theorem~\ref{Intro1} follows from Corollary~\ref{dichotomyPt2}; Theorem~\ref{Intro2} is exactly Corollary~\ref{FirstHalf}. 
The paradigm of a minimally unbounded $(P,\le,\delta)$
is $REF(\delta)$,  where $(P,\le)$ is itself an $\omega$-chain.  In this case, with Theorem~5.5 of \cite{URL},
$REF(\delta)$ is not Borel complete.  However, a typical minimally unbounded instance may have many additional $E_q$'s.  Here, under suitable large cardinal hypotheses,
we  show that if $(P, \leq, \delta)$ is minimally unbounded, then  $T_{P}$ is not Borel complete.

\begin{Fact}
    Suppose $(P, \leq, \delta) \in \PP$ is minimally unbounded. Then:

    \begin{enumerate}
        \item $\delta$ is bounded on every antichain (thus for all $Q \subseteq P$, $Q$ is bounded if and only if it is of bounded height).
        \item $P$ is narrow.
        \item $P$ admits an $\omega$-chain.
    \end{enumerate}
\end{Fact}
\begin{proof}
(1), (2): if they failed, then by the proof of Corollary~\ref{botCor} we would have some orthogonal, unbounded $Q_0, Q_1 \subseteq P$. But then $\dc(Q_0)$ would contradict that $P$ is minimally unbounded.

(3) $P$ is of unbounded height by (1), so this follows from Theorem~\ref{omegaChain} and (2).
\end{proof}

Now fix some minimally unbounded $(P, \leq, \delta)$ and fix an $\omega$-chain $(p_n: n < \omega)$. Let $R$ be the downward closure of this $\omega$-chain. For each $n\in\omega$, let $Q_n=\{q\in P:q\not > p_n\}$.
 
 The following Facts are easily verified from our assumption on $(P,\le,\delta)$.

 \begin{Fact} \label{Qnfacts}
 \begin{enumerate}
 \item  Each $Q_n$ is downward closed and bounded.
 \item  For each $n$, $Q_n\subseteq Q_{n+1}$ and $P=\bigcup_{n\in\omega} Q_n$.
 \end{enumerate}
     
 \end{Fact}

 \begin{proof}    (1)  Say $q\in Q_n$ and $q'\le q$.  If $q'\not\in Q_n$, then $q'>p_n$, hence $q> p_n$, contradicting $q\in Q_n$.    As $\{p_k:k\ge n+1\}\cap Q_n=\emptyset$, our hypothesis on $(P,\le,\delta)$ implies that $Q_n$ is bounded.

 (2)  Choose any $q\not\in Q_{n+1}$.  Then
 $q>p_{n+1}$, hence $q>p_n$, implying $q\not\in Q_n$.  Thus $Q_n\subseteq Q_{n+1}$.
 For the other half, by way of contradiction suppose there were some $r\in P\setminus \bigcup_{n\in\omega} Q_n$.  Then the infinite set $\{p_n:n\in\omega\}$ would be in $P_{\le r}$, contradicting our definition of $\PP$.

 \end{proof}

Our first goal is to prove the following `Schr\"{o}der-Bernstein' property for $\CC(\Phi)$ with respect to the notion of 1-embeddings defined in Definition~\ref{SBembed}.

%
%

 \begin{Theorem} \label{SB} Suppose $M,N\preceq \MM_P$ are countable.  Then for any colorings $(M,\cc), (N,\dd)$ of $M,N$, respectively, if there are  $f:(M, \cc)\preceq^*_1(N, \dd)$ and $g:(N, \dd) \preceq^*_1(M, \cc)$, then there is an $\LL(\cc)$-isomorphism $h:(M,\cc)\cong (N,\dd)$.
 \end{Theorem}

 Fix $(M,\cc)$ and $(N,\dd)$ as in Theorem~\ref{SB}. So there exist $1$-embeddings $f: (M, \cc) \preceq^*_1 (N, \dd)$ and $g: (N, \dd) \preceq^*_1 (M, \dd)$, but we do not fix choices of $f$ and $g$.

We shall need the following Lemma.

\begin{Lemma}\label{Stray}
    Suppose $K \models \Phi_{Q_n}$ and suppose $f: K \to K$ is an elementary embedding. Then there is some $k > 0$ such that $f^k = \mbox{id}$ (in particular $f$ is an automorphism).
\end{Lemma}
\begin{proof}
We can suppose $K \preceq \MM_{Q_n}$. As in the proof of Fact~\ref{check}(4) we can find some $g: \MM_{Q_n} \cong \MM_{Q_n}$ extending $f$. By Theorem~\ref{dichotomy}, $\Aut(\MM_{Q_n})$ has bounded exponent, so we can find some $k > 0$ such that $g^k = \mbox{id}$. Then necessarily $f^k = \mbox{id}$.
\end{proof}

For each $n$, recall from Section~\ref{QuotientsSection} that $[M]_{Q_n}$ is the quotient of $M$ by $E_{Q_n}$, which we view as an $\LL_{Q_n}$-structure, and similarly for $[N]_{Q_n}$; these are models of $\Phi_{Q_n}$, and whenever we have an $\LL_P$-embedding $f: M \to N$ we get an associated $\LL_{Q_n}$-embedding $[f]_{Q_n}: [M]_{Q_n} \to [N]_{Q_n}$.    Also, for every $a\in M$, let $[a]_{E_R}=\{a'\in M: E_R(a,a')\}$ denote the type definable $E_R=\bigwedge_{r\in R} E_r$-class of $a$.


\begin{Lemma}\label{Stray2}
    Suppose $f: (M, \cc) \preceq_1^* (N, \dd)$ 
    and $g:(N,\dd)\preceq^*_1 (M,\cc)$ are  $1$-embeddings. 
    \begin{enumerate}
    \item  For all $n\in\omega$, $[f]_{Q_n}:[M]_{Q_n}\rightarrow [N]_{Q_n}$ is an $\LL_{Q_n}$-isomorphism (in particular, $f_{Q_n}$ is onto); in fact, we can find $h: (N, \dd) \preceq^*_1 (M, \cc)$ with $[h]_{Q_n}$ and $[f]_{Q_n}$ inverse to each other. 
    \item  For every $a\in M$,  the restriction map $f\mr{[a]_{E_R}}:[a]_{E_R}\rightarrow [f(a)]_{E_R}$ is onto.
    \end{enumerate}
    
    Analogous statements hold for $g$.
\end{Lemma}
\begin{proof}
(1)  By Lemma~\ref{Stray} we can find $k > 0$ such that $([g \circ f]_{Q_n})^k = ([f \circ g]_{Q_n})^k = \mbox{id}$. Let $h = (g \circ f)^{k-1} \circ g$. Then $h$ is as desired. 

(2)  Choose any $a\in M$.  Since $f$ is an $\LL_P$ embedding, $f$ maps $[a]_{E_R}$ into $[f(a)]_{E_R}$.  To see that the restriction map is onto,  $f:(M,\cc)\preceq_1^* (N,\dd)$
being a  1-embedding implies that the corresponding $\LL_P$-embedding $f:N\rightarrow M$ is also a 1-embedding.  It follows that the restriction map $f\mr{[a]_{E_R}}:[a]_{E_R}\rightarrow 
[f(a)]_{E_R}$ is as well.  Since $P\setminus R$ is bounded by assumption, $g\mr{[a]_{E_R}}$ is onto by Proposition~\ref{useforbid}. 
\end{proof}

\begin{Definition} {\em
Suppose $f: (M, \cc) \preceq_1^* (N, \dd)$. Then let the {\em inverse} of $f$, denoted $f^{-1}$, be the natural partial elementary map from $(N, \dd)$ to $(M, \cc)$, so the domain of $f^{-1}$ is the range of $f$. For each $n < \omega$ let $[f^{-1}]_{Q_n}: [N]_{Q_n} \to [M]_{Q_n}$ denote the $\LL_{Q_n}$-isomorphism $[f]_{Q_n}^{-1}$ (as exists by the preceding lemma). Make similar definitions for $g: (N, \dd) \preceq_1^* (M, \cc)$.}
\end{Definition}

With these notions in hand, we now define a back-and-forth system between $(M,\cc)$ and $(N,\dd)$.  In the following, note that $\bigwedge_{m < n} E_{p_m}$ is either the indiscrete equivalence relation if $n = 0$, or else $E_{p_{n-1}}$. 

\begin{Definition} \label{simdef}
    Suppose $\overline{a} \in M, \overline{b} \in N$ are tuples of the same length $i_*$. Then say that $\overline{a} \sim \overline{b}$ if there exist $(f_i, g_i)_{i < i_*}$ such that the following conditions all hold:

    \begin{enumerate}
        \item For each $i < i_*$, either $f_i: (M, \cc) \preceq_1^* (N, \dd)$ with inverse $g_i$ (so $g_i$ need not be total), or else $g_i: (N, \dd) \preceq_1^* (M, \cc)$ with inverse $f_i$ (so $f_i$ need not be total);
        \item Each $f_i(a_i) = b_i$ and $g_i(b_i)=a_i$;
        \item For all $i,j<i^*$ and for all $n < \omega$, if $M\models \bigwedge_{m<n} E_{p_m}(a_i,a_j)$ then $[f_i]_{Q_n} = [f_j]_{Q_n}$;
        \item For all $i,j<i^*$ and for all $n < \omega$, if $N\models \bigwedge_{m<n} E_{p_m}(b_i, b_j)$ then $[g_i]_{Q_n} = [g_j]_{Q_n}$.
    \end{enumerate}
\end{Definition}

Note that in the above, if $M\models E_R(a_i, a_j)$, then
as $R$ is the downward closure of $(p_n:n\in\omega)$,
we must have $f_i = f_j$ by (3).

\medskip
\noindent{\bf Claim 1.} 
Suppose $\overline{a} \in M, \overline{b} \in N$ are tuples of the same length $i_*$, and $(f_i, g_i)_{i < i_*}$ satisfy conditions (1), (2), and (3) in the definition of $\sim$. Then $\mbox{qftp}_{\LL_P(c)}(\overline{a}) = \mbox{qftp}_{\LL_P(c)}(\overline{b})$ and Clause~(4) holds as well.

\begin{proof}
Clearly each $\cc(a_i) = \dd(b_i)$ since $f_i$ is color-preserving.   Next, choose $i,j<i^*$.  We  show that
 $M\models E_p(a_i, a_j)$ if and only if $N\models E_p(b_i, b_j)$
 for every $p\in P$ by splitting into cases.   First, if $E_R(a_i,a_j)$ holds, then as $f_i=f_j$ we are done since
 $f_i$ preserves quantifier-free types.  If $E_R(a_i,a_j)$ fails,
 then choose $n$ to be 
least such that $M\models \neg E_{p_n}(a_i, a_j)$. Then by (3), 
 $[f_i]_{Q_n}([a_i]_{Q_n})=[b_i]_{Q_n}$ and $[f_i]_{Q_n}([a_j]_{Q_n})=[b_j]_{Q_n}$.  
  It follows that for all $p\in Q_n$, $M\models E_p(a_i, a_j)$ if and only if $N\models E_p(b_i, b_j)$. In particular this holds holds for $p= p_{n}$. By choice of $n$, we have that $M\models \lnot E_{p_n}(a_i, a_j)$, so $N\models \lnot E_{p_n}(b_i, b_j)$
  as well.  
  Thus, whenever $p > p_n$, both $M\models \lnot E_{p}(a_i, a_j)$ and $N\models\lnot E_p(b_i, b_j)$, so $E_p$ is preserved in all cases.  Thus, $\mbox{qftp}_{\LL_P(c)}(\overline{a}) = \mbox{qftp}_{\LL_P(c)}(\overline{b})$.  
  Clause~(4) follows from this since $f$ and $g$ are inverses.
\end{proof}

The existence of an  isomorphism $h:(M,\cc)\rightarrow (N,\dd)$ as in Theorem~\ref{SB}, 
follows immediately from our second claim.

\medskip\noindent{\bf Claim 2.} 
    $\sim$ describes  a back-and-forth system from $(M, \cc)$ to $(N, \dd)$. 

\begin{proof}
We have already established that $\sim$ preserves quantifier-free type, and it follows from the definition that $\sim$ is symmetric in $M$ and $N$, so it is enough to verify the forth condition. So suppose $\overline{a} \sim \overline{b}$ via $(f_i, g_i: i < i_*)$ and suppose $a_{i_*}$ is given. There are four cases. 

{\bf Case 1.} Suppose $\overline{a}$ is empty, i.e. $i_* = 0$. Then let $f_0: (M, \cc) \preceq_1^* (N, \dd)$ be any $1$-embedding, and this witnesses $a_0 \sim f_0(a_0)$.

{\bf Case 2.} Suppose $E_R(a_{i_*}, a_i)$ for some $i < i_*$. Then let $f_{i_*} = f_i, g_{i_*} = g_i$.   If $f_i$ was total,
then obviously $a_{i^*}\in\dom(f_i)$, so put $b_{i_*}:=f_i(a_{i^*})$.
On the other hand, if  $f_i$ is not total, then $g_i$ is total.  Then, applying Lemma~\ref{Stray2}(2) to $g_i$, we get that $a_{i^*}=g_i(b_{i_*})$ for some $b_{i_*}\in [b_i]_{E_R}$.   
In either case, we have $\overline{a} a_{i_*} \sim \overline{b} b_{i_*}$.

Assuming we are not in Case 1 nor Case 2,  let $n$ be maximal  such that there is $i < i_*$ with $M\models \bigwedge_{m<n} E_{p_m}(a_{i_*}, a_i)$. Fix such an $i$. 

{\bf Case 3.} If $a_{i_*} \in \mbox{dom}(f_i)$ then let $f_{i_*} = f_i, g_{i_*} = g_i$ and let $b_{i_*} = f_i(a_{i_*})$, and then $\overline{a} a_{i_*} \sim \overline{b} b_{i_*}$.

{\bf Case 4.} Suppose $a_{i_*} \not \in \mbox{dom}(f_i)$. Then $g_i: (N, \cc) \preceq_1^* (M, \dd)$  is total by Definition~\ref{simdef}(1).  By Lemma~\ref{Stray2}(1) we can find $f_{i_*}: (M, \cc) \preceq_* (N, \dd)$ with $[f_{i_*}]_{Q_n}$ and $[g_{i}]_{Q_n}$ inverse to each other. Let $g_{i_*} = f_{i_*}^{-1}$ and let $b_{i_*} = f_{i_*}(a_{i_*})$. Then this works.
\end{proof}

With the proof of Theorem~\ref{SB} in hand, 
    at the cost of introducing a large cardinal, we obtain 
our final result.  We first explain what large cardinal we will be using, namely an $\omega$-Erd\H{o}s cardinal:

\begin{Definition}  \label{Erdos}  {\em 
	Suppose $\alpha$ is an ordinal (we will only use the case $\alpha = \omega$). Then let $\kappa(\alpha)$ be the least cardinal $\kappa$ with $\kappa \rightarrow (\alpha)^{<\omega}_2$ (if it exists). In words: whenever $F: [\kappa(\alpha)]^{<\omega} \to 2$, there is some $X \subseteq \kappa(\alpha)$ of ordertype $\alpha$, such that $F \mr{[X]^n}$ is constant for each $n < \omega$.}
	
\end{Definition}

The cardinal $\kappa(\omega)$ is a large cardinal: it is always inaccessible and has the tree property. On the other hand, it is absolute to $\mathbb{V} = \mathbb{L}$, and is well below the consistency strength of a measurable cardinal. See \cite{Silver} for a description of these results.

In \cite{SB} the second author defines the notion of thickness. For every $\Phi \in \mathcal{L}_{\omega_1 \omega}$ and for every infinite cardinal $\lambda$ we have the thickness $\tau(\Phi, \lambda)$ of $\Phi$ at $\lambda$, a cardinal invariant of the complexity of $\Phi$. The second author also defines what it means for $(\Phi, \sim_{\alpha \omega})$ to have the Schr\"{o}der-Bernstein property. Then we have:
\begin{Theorem}
    \begin{enumerate}
    \item (Follows from Theorem 5.8 of \cite{SB}) For all $\Phi, \Psi \in \mathcal{L}_{\omega_1 \omega}$, if $\Phi \leq_B \Psi$ then for all $\lambda$, $\tau(\Phi, \lambda) \leq \tau(\Psi, \lambda)$;
    \item (Corollary 5.13 of \cite{SB}) If TAG denotes torsion abelian groups then $\tau(\mbox{TAG}, \lambda) = \beth_1(\lambda)$ whenever $\lambda$ is inaccessible or $\aleph_0$.
        \item (Theorem 11.8 of \cite{SB}) Suppose $\kappa(\omega)$ exists,  $\alpha < \kappa(\omega)$, and $(\Phi, \sim_{\alpha \omega})$ has the Schr\"{o}der--Bernstein property. Then for all $\lambda$, $\tau(\Phi, \lambda) \leq \lambda^{<\kappa(\omega)}$. In particular, $\tau(\Phi, \kappa(\omega)) \leq \kappa(\omega)$. 
    \end{enumerate}
    	
\end{Theorem}

Putting these together we get 

\begin{Corollary}  \label{TAGcor}
    Suppose $\kappa(\omega)$ exists, and $\alpha < \kappa(\omega)$, and $(\Phi, \sim_{\alpha \omega})$ has the Schr\"{o}der--Bernstein property. Then TAG is not Borel reducible to $\Phi$, and so $\Phi$ is not Borel complete.
\end{Corollary}

Our final theorem follows immediately.

\begin{Theorem}  \label{SecondHalf}
Suppose $(P, \leq, \delta) \in \PP$ is minimally unbounded. If $\kappa(\omega)$ exists then $T_P$ is not Borel complete.
\end{Theorem}

\begin{proof}  Suppose $(P, \leq, \delta) \in \PP$ is minimally unbounded and $\kappa(\omega)$ exists. We claim that $(\CC(\Phi_P), \sim_{1 \omega})$ has the Schr\"{o}der-Bernstein property, which suffices by Corollary~\ref{TAGcor} and Lemma~\ref{shift1}. Let $\preceq_{1, \omega}$ be defined like $\preceq_1^*$, except replacing quantifier-free type by first-order type. We need to show that for all countable $(M, \cc), (N, \dd) \models \CC(\Phi_P)$, if there exist $f: (M, \cc) \preceq_{1 \omega} (N, \dd)$ and $g: (N, \dd) \preceq_{1 \omega} (M, \cc)$ then $(M, \cc) \cong (N, \dd)$. This follows at once from Theorem~\ref{SB}, using the trivial fact that $\preceq_{1 \omega}$ is a stronger notion than $\preceq_1^*$.
\end{proof}

\section{Questions}

The first question is the most obvious:

\medskip

\noindent \textbf{Question.} Can we remove the large cardinal? This would involve getting a better handle on the minimally unbounded case.

\medskip

We also ask:

\medskip
\noindent \textbf{Question.} Which mutually algebraic theories are Borel complete? What if we restrict attention to tame expansions of $REF(bin)$?

\appendix

\section{Classifying the reducts of models of REF}

In this Appendix, we concentrate on refining equivalence relations, possibly with infinite splitting.  Fix a language $\LL=\{E_n:n\in \omega\}$ and let $REF$ be the (incomplete) theory
asserting that each $E_n$ is an equivalence relation; $E_{n+1}$ refines $E_n$, i.e., $REF\vdash \forall x\forall y (E_{n+1}(x,y)\rightarrow E_n(x,y))$;
and there is a function $\delta:\omega\rightarrow\{2, 3, \ldots, \infty\}$ specifying the number of classes  $E_{n+1}$ partitions each $E_n$-class into.    For the Appendix, we allow infinite splitting.  
We classify the reducts of any model $M\models REF$, but first we formally define a reduct of a structure.  

\begin{Definition}  {\em
Given a non-empty set $M$ and two families $\A=\{D_i:i\in I\}$ and $\A'=\{D_j':j\in J\}$ of sets of subsets of $M^{k(j)}$ for various $k(j)$,
we say that $\A$ and $\A'$ are {\em $\emptyset$-definably equivalent} if
every $D_j'$ is $\emptyset$-definable in the structure $(M,D_i:i\in I)$, and every $D_i$ is $\emptyset$-definable in the structure $(M,D_j':j\in J\}$.
}
\end{Definition}

Evidently, if $\A$ and $\A'$ are $\emptyset$-definably equivalent, then the two structures above have the same definable sets, either with or without parameters.

\begin{Definition}  {\em  Suppose $M$ is an $\LL$-structure and suppose  $\{D_i:i\in I\}$ is any set of $\emptyset$-definable subsets of $M^{k(i)}$ for various $k(i)$.
Any such set defines a {\em reduct} of $M$.
}
\end{Definition}

\medskip\centerline{{\bf  Throughout this Appendix, `definable' always means $\emptyset$-definable.}}
\medskip

An obvious class of reducts of a model of REF have the form $\LL_I:=\{E_n:n\in I\}$ for an arbitrary subset $I\subseteq \omega$.  Our theorem is that, up to $\emptyset$-definable equivalence,
these are the only reducts of any $M\models REF$.

\begin{Theorem} \label{app}  Let $M\models REF$ be arbitrary and let $\A=\{D_i:I\in I\}$ be a set of definable subsets of $M^{k_i}$ for various $k_i$.
Then there is a subset $J\subseteq \omega$ such that $\A$ and $\{E_j:j\in J\}$ are $\emptyset$-definably equivalent.
\end{Theorem}

Towards a proof of Theorem~\ref{app}, first note that it suffices to prove this when $\A$ is a singleton, i.e., we are given a single definable subset 
$D(\xbar)\subseteq M^{\lg(\xbar)}$.  Second, it will simplify the notation to reindex the original language $\LL$ to include two additional equivalence relations.
We insist that $E_0$ is the equivalence relation on $M$ with only one class, i.e., $M\models \forall x\forall y E_0(x,y)$, and we add a new equivalence relation $E_\omega$
for equality, i.e., $M\models \forall x \forall y (E_\omega(x,y)\leftrightarrow x=y)$.  

What we will really prove is the following Theorem, which immediately yields Theorem~\ref{app}. We phrase it this way to facilitate a delicate induction.

\begin{Theorem}  \label{apprevise}  Let $M\models REF$ and let $D(x_0,\dots, x_{n-1})\subseteq M^n$ be definable.  Then there is a finite set
$F\subseteq \omega+1$ such that $\{D,=\}$ and $\{E_j:j\in F\}$ are $\emptyset$-definably equivalent.
\end{Theorem}

The idea of the proof of Theorem~\ref{apprevise} is as follows.  Clearly, any definable $D\subseteq M^n$ mentions only finitely many $E_j$'s, so there is a finite
subset $u\subseteq \omega+1$ for which $D$ is definable in $M_u$, the reduct of $M$ to the language $\LL_u:=\{E_j:j\in u\}$.  Without loss, we may assume
$\{0,\omega\}\subseteq u$.  This is `half' of what we require.  We also need to show that we can `whittle away' unnecessary $E_j$'s, still maintaining that $D$ is definable,
so that each of the resulting $E_j$'s are $\{D,=\}$-definable.

Our toolbox makes heavy use of the fact that  for any finite $u\subseteq \omega+1$ with $0\in u$, the reduct $M_u$ of $M$ to  
$\LL_u=\{E_j:j\in u\}$ is an $\omega$-categorical structure. 
First, we conclude that a subset $D\subseteq M_u^k$ is $\LL_u$-definable if and only if $D$ is invariant under every $\LL_u$-automorphism of $M_u$.
As well,  there are only finitely many 2-types, each of which is isolated.  Moreover, we can explicitly describe these 2-types and their isolating formulas:

$x = y$ is one complete type. 
If $k=\max(u) < \omega$, then $E_k(x,y) \land x \not= y$ isolates another complete type.  
For each $i\in u$ with $i<k$, let $i^+$ denote the immediate successor of $i$ in $u$.  
Then for each such $i$, there is a complete 2-type $p_i(x,y)$ isolated by $E_i(x,y)\wedge\neg E_{i+1}(x,y)$.  
[Note that since $0\in u$ and $E_0(x,y)$ always holds, this list is exhaustive.]
We introduce one technical concept.  

\begin{Definition}  {\em  Say $D(\xbar)$ is $\LL_u$-definable with $k=\max(u)$.  We say $D$ is {\em $u$-irreflexive} if $D(\xbar)\vdash \bigwedge_{i\neq j} \neg E_k(x_i,x_j)$.
}
\end{Definition}

\begin{Lemma}  \label{technical}  Suppose $D(\xbar)$ is $\LL_u$-definable with $k=\max(u)$.  Let $w=u\setminus\{k\}$ and suppose $D(\xbar)$ is $w$-irreflexive.
Then $D(\xbar)$ is $\LL_w$-definable.
\end{Lemma}

\begin{proof}  We show that $D$ is invariant under every $\LL_v$-automorphism $\sigma\in \Aut(M_v)$.  So choose $\abar$ with $D(\abar)$ holding and let $\bbar=\sigma(\abar)$.
Since $D$ is $w$-irreflexive, we have that $\bigwedge_{i\neq j} \neg E_\ell(a_i,a_j)$, where $\ell=\max(w)$.  As $\sigma$ preserves $E_\ell$, we also have
$\bigwedge_{i\neq j} \neg E_\ell(b_i,b_j)$.  As $E_k$ refines $E_\ell$, we conclude that $\bigwedge_{i\neq j} \neg E_k(a_i,a_j)$ and $\bigwedge_{i\neq j} \neg E_k(b_i,b_j)$.
Thus, $\sigma$ also preserves $E_k$ on $\abar,\bbar$, so $\tp_{\LL_u}(\abar)=\tp_{\LL_u}(\bbar)$.
Since $D(\overline{x})$ is $\LL_u$-definable and $D(\overline{a})$ holds, we get $D(\overline{b})$ holds as well.
\end{proof}

The following Proposition provides the crux of the induction used in proving Theorem~\ref{apprevise}.

\begin{Proposition} \label{MainLemma} Suppose  $u\subseteq \omega+1$ is finite with $0\in u$ and $\ell<k$ the two largest elements of $u$.  Suppose $D(\xbar)$ is $\LL_u$-definable and $u$-irreflexive.
If $E_\ell$ is not $\{D,E_k\}$-definable, then $D$ is $\LL_{u\setminus \{\ell\}}$-definable.  
\end{Proposition}

\begin{proof}   Because of our trivial interpretation of $E_0$, the Proposition is easy if $|u|=2$,  so assume $|u|\ge 3$ and let $m<\ell<k$ denote the top three elements in $u$.
Let $v=u\setminus\{\ell\}$.
We begin by defining a single $\{D,E_k\}$-definable formula $\phi(x,y)$ and we actually prove the stronger statement that if $\phi$ is not equivalent to $E_\ell$,
then $D$ is $\LL_v$-definable.  


For $\rho\in \Sym(n)$, let $D^\rho$ be the same formula $D(x_{\rho(0)},\dots,x_{\rho(n-1)})$, whose free variables are permuted by $\rho$.  
Let $$\theta(x,y):=\bigwedge_{\rho\in \Sym(n)} \forall z_2\dots\forall z_{n-1}[ D^\rho(x,y,z_2,\dots,z_{n-1})\leftrightarrow D^\rho(y,x,z_2,\dots,z_{n-1})]$$

Let $\Edist_k(z_0,\dots,z_{n-1}):=\bigwedge_{i\neq j}  \neg E_k(x_i,x_j)$ and let

$$\psi(x,y):=\bigwedge_{\rho\in \Sym(n)}\forall z_1\dots\forall z_{n-1} \left(\Edist_k(x,\zbar)\wedge\Edist_k(y,\zbar)\rightarrow [D^\rho(x,\zbar)\leftrightarrow D^\rho(y,\zbar)]\right)$$

Finally, put $\phi(x,y):=\theta(x,y)\wedge \psi(x,y)$.

The proof  of Proposition~\ref{MainLemma} follows immediately from the following three claims.

\medskip
\noindent
{\bf Claim 1.}  $E_\ell(x,y)\vdash \phi(x,y)$.

\begin{proof}  Choose any $a,b\in M$ with $E_\ell(a,b)$.  We first show $\theta(a,b)$, with the verification of $\psi(a,b)$ similar and left to the reader.
Choose any $c_2,\dots,c_{n-1}$ from $M$.  By symmetry, it suffices to show that $D^\rho$ is preserved when $\rho=id$, i.e., when $D^\rho=D$.
If both $D(a,b,\cbar)$ and $D(b,a,\cbar)$ fail we are done, so suppose $D(a,b,c_2,\dots,c_{n-1})$ holds.  Since $D$ is $u$-irreflexive, it follows that $\Edist_k(a,b,c_2,\dots,c_n)$.
In particular,  $\neg E_k(a,c_j)$ and $\neg E_k(b,c_j)$ hold for 
all $2\le j\le n-1$.  
 Since $E_\ell(a,b)$, we also  have $E_\ell(a,c_j)\leftrightarrow E_\ell(b,c_j)$  for each $2\le j\le n-1$.  It follows that $\tp_{\LL_u}(a,b,c_2,\dots,c_{n-1})=\tp_{\LL_u}(b,a,c_2,\dots,c_{n-1})$.  As $D$ is $\LL_u$-definable and $D(a, b, c_2, \dots, c_{n-1})$ holds, we conclude that $D(b,a,c_2,\dots,c_n)$.  Thus, $\theta(a,b)$ holds.
 That $\psi(a,b)$ holds is similar.
 \end{proof}

 \medskip
\noindent
{\bf Claim 2.}  If the $\{D,E_k\}$-formula $\phi(x,y)$ does not define $E_\ell$, then $E_m(x,y)\vdash \phi(x,y)$.

\begin{proof}  By Claim 1, we know that $E_\ell(x,y)\vdash \phi(x,y)$, so if $\phi$ does not define $E_{\ell}$, then there are $a,b\in M_u$
such that $\phi(a,b)\wedge\neg E_\ell(a,b)$ holds.  Let $p_i(x,y)=\tp_{\LL_u}(a,b)$.  As $p_i$ is a complete type isolated
by $E_i(x,y)\neg E_{i^+}(x,y)$, we conclude that
$$E_i(x,y)\wedge \neg E_{i^+}(x,y)\vdash \phi(x,y)$$
[Here, $i^+$ is the least element of $u>i$.]  Since $\neg E_\ell(a,b)$ holds, $i<\ell$, hence $i\le m$.  Since $E_m$ refines $E_i$, Claim~3 will be established once we prove
the following Subclaim.   

\medskip
\noindent
{\bf  Subclaim.}  $E_i(x,y)\vdash \phi(x,y)$.

\begin{proof}  We first show $E_i(x,y)\vdash \theta(x,y)$.  By symmetry, it suffices to show this for $\rho=id$, i.e., when $D^\rho=D$.
So choose $a,b,\cbar$ from $M$ with $E_i(a,b)$.  Since we already know that $p_i(x,y)\vdash\phi(x,y)$, we may additionally assume that $E_{i^+}(a,b)$ holds as well.
Clearly, if $\Edist_k(a,b,\cbar)$ fails, then since $D$ is $u$-irreflexive, $\neg D(a,b,\cbar)\wedge \neg D(b,a,\cbar)$ hold and we are done, so
assume $\{a,b,\cbar\}$ are pairwise $E_k$-inequivalent.
There are now two subcases.  First, assume there is some $c_j\in\cbar$ such that $p_i(a,c^*)$ holds.  Then also $p_i(b,c_j)$ holds.  To ease notation, suppose $j=2$ and 
write $\dbar$ for $(c_3,\dots,c_{n-1})$.   Then, as $p_i(x,y)\vdash\theta(x,y)$, we have
$$D(a,b,c_2,\dbar)\leftrightarrow D(c_2,b,a,\dbar)\leftrightarrow D(b,c_2,a,\dbar)\leftrightarrow D(b,a,c_2,\dbar)$$
as needed. [The three equivalences follow from $\theta(a,c_2)$, $\theta(b,c_2)$, and $\theta(c_2,a)$, respectively.]  

So, assume there no such $c_j\in\cbar$.  By the homogeneity of $M_u$, choose some $a^*\in M$ such that $p_i(a,a^*)$ holds.  As $E_{i^+}(a,b)$ hold, $p_i(b,a^*)$ holds as well.
Since $\{a,b,\cbar\}$ are $E_k$-inequivalent and since we are in this subcase, we also have $\neg E_k(c_j,a^*)$ for all $2\le j\le n-1$.  Thus, $\{a,b,a^*,\cbar\}$ are all $E_k$-inequivalent.
Since $p_i\vdash\psi(x,y)$, we can exchange $a^*$ with either of $a$ or $b$ in any coordinate.  
Thus, $$D(a,b,\cbar)\leftrightarrow D(a^*,b,\cbar)\leftrightarrow D(b,a^*,\cbar)\leftrightarrow D(b,a,\cbar)$$
using $\psi(a,a^*)$, $\theta(a^*,b)$, and $\psi(a,a^*)$, respectively.  So we have established that $E_i(x,y)\vdash \theta(x,y)$.

Finally, we show $E_i(x,y)\vdash \psi(x,y)$.  For this, choose $a,b\in M$ such that $E_i(a,b)$.  As we know $p_i\vdash\psi(x,y)$, again we may assume  $E_{i^+}(a,b)$ as well.
We split into the same two subcases as before.  First, assume there is some $c_j\in\cbar$ such that $p_i(a,c_j)$.  For notational simplicity, assume $j=1$ and write
$\cbar=c_1\dbar$.  Then also, $p_i(b,c_1)$, so we have $\psi(a,c_1)$ and $\psi(b,c_1)$.
Thus, 
$$D(a,c_1,\dbar)\leftrightarrow D(a,b,\dbar)\leftrightarrow D(b,a,\dbar)\leftrightarrow D(b,c_1,\dbar)$$
with the second equivalence using the implication $E_i(x,y)\vdash\theta(x,y)$ established above. 
Finally, assume that there is no $c_j\in\cbar$ with $p_i(a,c_j)$.  Similarly to the $\theta$ case, choose $a^*\in M$ such that $p_i(a,a^*)$, hence also $p_i(b,a^*)$.
By our case assumption, $\{a,b,a^*,\cbar\}$ are pairwise $E_k$-inequivalent, so as $p_i\vdash\psi$, we can swap $a^*$ for either $a$ or $b$.
Thus, $$D(a,\cbar)\leftrightarrow D(a^*,\cbar)\leftrightarrow D(b,\cbar)$$
completing the proof of the Subclaim and hence of Claim 2.
\end{proof}
\end{proof}

\medskip
\noindent
{\bf Claim 3.}  If $E_m(x,y)\vdash \phi(x,y)$ then $D$ is $\LL_v$-definable, where $v=u\setminus\{\ell\}$.

\begin{proof}  We argue that $D$ is preserved under $\LL_v$-automorphisms $\sigma\in\Aut(M)$.  So choose $\abar\in M^n$ such that $D(\abar)$ holds and let $\bbar=\sigma(\abar)$.
We will conclude that $D(\bbar)$ holds by constructing, in $n$ steps, some $\cbar\in M^n$ such that $D(\cbar)$ holds and $\tp_{\LL_u}(\cbar)=\tp_{\LL_u}(\bbar)$.
We will inductively construct $n$-tuples $\abar^i=(a^i_j:j<n)$ for $i<n$ from $M$ satisfying
\begin{enumerate}
\item  $\abar^i\equiv_{\LL_v} \bbar$;
\item  $D(\abar^i)$ holds; and
\item  For all $j,j'<i$, $E_\ell(a^i_j,a^i_{j'})\leftrightarrow E_\ell(b_j,b_{j'})$.
\end{enumerate}
If we can succeed, we put $\cbar:=\abar^n$ and we finish.  Note that (1) and (3) combine to say that the $i$-element subtuples
$(a^i_j:j<i)$ and $(b_j:j<i)$ have the same $\LL_u$-type.
So, put $\abar^0:=\abar$ and assume $i<n-1$ and $\abar^i$ has been defined and satisfies the three requirements.
There are now three cases about how we choose $\abar^{i+1}$.

First, if $\neg E_m(a^i_j,a^i_i)$ for all $j<i$, then simply put $\abar^{i+1}:=\abar^i$ and all three conditions are met, using that  $E_\ell$ refines $E_m$.
So, assume this is not the case.  Fix $j^*<i$ such that $E_m(a^i_{j^*},a^i_i)$ holds.  Since $\abar^i\equiv_{\LL_v} \bbar$ holds, it follows that $E_m(b_i,b_{j^*})$ holds as well.
As $M_u$ is homogeneous, choose $a^*\in M$ so that the $(i+1)$-tuples $(a^i_j:j<i,a^*)\equiv_{\LL_u}(b_j:j\le i)$.   Note that $E_m(a^*,a_{j^*})$ iff $E_m(b_i,b_j^*)$, so from our
choice of $j^*$, we also have $E_m(a^i_i,a^i_{j^*})$, hence $E_m(a^*,a^i_i)$ by transitivity.   By our assumption, $\phi(a^*,a^i_i)$ holds.  
Our second case is to assume that  $\Edist_k(a^*,a_j:j\neq i)$ holds.  Here, let $\abar^{i+1}$ be the sequence formed
 from $\abar^i$ by exchanging $a^i_i$ by $a^*$.  
 Since $\psi(a^*,a^i_i)$ and $D(\abar^i)$ both hold, we have $D(\abar^{i+1})$ holding as needed.
 
The remaining case is where $\Edist_k(a^*,a_j:j\neq i)$ fails.  In this case, choose $j^{**}$ such that $E_k(a^*,a^i_{j^{**}})$ holds. The choice of $a_*$ implies that $j^{**} \geq i$. Then the $(i+1)$-tuples
$(a^i_j:j<i,a^i_{j^{**}})\equiv_{\LL_u} (b_j:j\le i)$, so take $\abar^{i+1}$ to be $\abar^i$, with the elements $a^i_i$ and $a^i_{j^{**}}$ interchanged.  
Since $E_m(a^*, a^i_i)$ holds and $E_k(a^*,a^i_{j^{**}})$ holds, we have $E_m(a^i_i,a^i_{j^{**}})$ holds, hence $\theta(a^i_i,a^i_{j^{**}})$.
Thus, $D(\abar^{i+1})$ as required.
\end{proof}
These three Claims finish the proof of Proposition~\ref{MainLemma}.

\end{proof}

\begin{Definition}  {\em  For $n=\{0,\dots,n-1\}$, a {\em projection} $\pi:n\rightarrow R$ is any function satisfying  $\pi(r)=r$ for all $r\in R$ (so $R\subseteq n$).
As notation, $\pi[n]$ denotes the image of $\pi$.  For any $k\in\omega+1$ and for any projection $\pi:n\rightarrow R$, let
$$\Delta^\pi_k(x_0,\dots,x_{n-1}):=\bigwedge_{\pi(i)=\pi(j)} E_k(x_i,x_j)\wedge\bigwedge_{\pi(i)\neq \pi(j)} \neg E_k(x_i,x_j)$$
and for a given formula $D(x_0,\dots,x_{n-1})$, 
$$D^\pi_k(x_i:i\in \pi[n]):=\exists y_0\dots\exists y_{n-1} [\bigwedge_{i\in \pi[n]} y_i=x_i \wedge \Delta^\pi_k(y_0,\dots,y_{n-1}) \wedge  D(y_0,\dots,y_{n-1})]$$
}
\end{Definition}

\medskip
\noindent
{\bf Proof of Theorem~\ref{apprevise}.}

\medskip

  Given an $\LL$-definable $D(x_0,\dots,x_{n-1})$, choose a finite $u\subseteq \omega+1$ with
$\{0,\omega\}\subseteq u$ such that $D$ is $\LL_u$-definable.  Note that
$D(\xbar)\leftrightarrow \bigvee_\pi D^{\pi}_\omega$, where the disjunction ranges over all projections of $n$.   As each $D^\pi_\omega$ is $u$-irreflexive,
we may assume that our original $D(\xbar)$ is $u$-irreflexive.  Thus, the Theorem is proved, once we establish the following Claim, which is proved by induction on $|u|$. 

\medskip
\noindent{\bf Claim.}  Suppose $u\subseteq \omega+1$ is finite with $0\in u$, and let $k=\max(u)$.  If $D(\xbar)$ is $\LL_u$-definable and $u$-irreflexive, then
there is some $F\subseteq u$ such that the sets $\{D,E_k\}$ and $\{E_j:j\in F\}$ are $\emptyset$-definably equivalent.

\begin{proof}  As noted above, we argue by induction on $|u|$.  Say $|u|=m$ and the result holds for all proper subsets of $u$.  Choose $D(\xbar)$ to be $u$-definable and $u$-irreflexive.  There are now two cases:

\medskip
\noindent
{\bf Case 1.}  $E_\ell$ is not $\{D,E_k\}$-definable. 

\medskip

In this case, by Proposition~\ref{MainLemma}, taking $v=u\setminus\{\ell\}$,  $D$ is $\LL_v$-definable and $v$-irreflexive.
Thus, we finish by the inductive hypothesis.

\medskip
\noindent
{\bf Case 2.}  $E_\ell$ is $\{D,E_k\}$-definable.  

\medskip

Here, note that

$$D(\xbar)\leftrightarrow \bigwedge_{i\neq j} \neg E_k(x_i,x_j)\wedge\bigvee_\pi D^\pi_\ell(x_i:i\in \pi[n])$$
so the sets $\{D,E_k\}$ and  $\{D^\pi_\ell:\pi$ a projection$\}\cup\{E_\ell,E_k\}$
are definably equivalent.
It follows that each $D^\pi_\ell$ is $\LL_u$-definable (since $D$ was) and is visibly $w$-irreflexive, where $w=u\setminus\{k\}$.  
By Lemma~\ref{technical}, we conclude that each $D^\pi_\ell$ is $\LL_w$-definable, so we can apply the inductive hypothesis on each $D^\pi_\ell$.
For each $\pi$, choose a finite set $F^\pi\subseteq w$ so that $\{D^\pi_\ell,E_\ell\}$ and $\{E_j:j\in F_\phi\}$ are $\emptyset$-definably equivalent.
Put $F^*=\bigcup_\pi F_\pi\cup \{E_\ell,E_k\}$.  It follows that $\{D,E_k\}$ and $\{E_j:j\in F^*\}$ are $\emptyset$-definably equivalent.
\end{proof}

\end{document}